\documentclass[10pt]{amsart}
\usepackage{amssymb}
\usepackage{color}
\usepackage{graphicx}
\usepackage{float}
\usepackage[all,cmtip]{xy}
\usepackage{tikz}
\usetikzlibrary{matrix}
\usepackage{url}
\usepackage{subfig}
\usepackage{hyperref}
\usepackage{textcomp}

\newcommand{\circled}[1]{\textcircled{\scriptsize #1}}

\newcommand{\rot}{\operatorname{rot}}

\mathchardef\mhyphen="2D

\newcommand{\std}{{\operatorname{std}}}

\newcommand{\Sat}{\operatorname{Sat}}
\newcommand{\Fat}{\operatorname{Fat}}

\newcommand{\im}{{\operatorname{Image}}}
\newcommand{\Id}{{\operatorname{Id}}}

\newcommand{\Rot}{\operatorname{Rot}}
\newcommand{\tb}{\operatorname{tb}}

\newcommand{\R}{{\mathbb{R}}}

\newcommand{\Z}{{\mathbb{Z}}}
\newcommand{\N}{{\mathbb{N}}}

\newcommand{\NS}{{\mathbb{S}}}
\newcommand{\D}{{\mathbb{D}}}

\newcommand{\Leg}{{\mathfrak{Leg}}}
\newcommand{\FatLeg}{{\mathfrak{FatLeg}}}

\newcommand{\kalman}{{K\'{a}lm\'{a}n }}
\newcommand{\kalmans}{{K\'{a}lm\'{a}n's }}

\newcommand{\LL}{{\mathcal{L}}}

\usepackage[framemethod=TikZ]{mdframed}
\newcounter{theo}[section]

\newtheorem{proposition}{Proposition}
\newtheorem{theorem}[proposition]{Theorem}
\newtheorem{definition}[proposition]{Definition}
\newtheorem{lemma}[proposition]{Lemma}

\newtheorem{corollary}[proposition]{Corollary}
\newtheorem{remark}[proposition]{Remark}

\newtheorem{example}[proposition]{Example}

\makeatletter
\newcommand{\superimpose}[2]{%
  {\ooalign{$#1\@firstoftwo#2$\cr\hfil$#1\@secondoftwo#2$\hfil\cr}}}
\makeatother

\emergencystretch=\maxdimen
\begin{document} 

\title{Parametric satellites and connected-sums in the space of Legendrian embeddings}

\subjclass[2020]{Primary: 53D10.}
\date{\today}

\keywords{}

\author{Eduardo Fern\'{a}ndez}
\address{Department of Mathematics, University of Georgia, Athens, GA, USA}
\email{eduardofernandez@uga.edu}

\author{Javier Mart\'{i}nez-Aguinaga}
\address{Universidad Complutense de Madrid, Departamento de Algebra, Geometría y Topología, Facultad de
Matemáticas, and Instituto de Ciencias Matemáticas CSIC-UAM-UC3M-UCM, C. Nicolás Cabrera, 13-15,
28049 Madrid, Spain}
\email{frmart02@ucm.es}

\author{Francisco Presas}
\address{Instituto de Ciencias Matem\'{a}ticas CSIC-UAM-UC3M-UCM, C. Nicol\'{a}s Cabrera, 13-15, 28049 Madrid, Spain.}
\email{fpresas@icmat.es}

\begin{abstract}

This article introduces two new constructions at the higher homotopy level in the space of Legendrian embeddings in $(\R^3,\xi_{\std})$. We first introduce the parametric Legendrian satellite construction, showing that the satellite operation works for parametric families of Legendrian embeddings. This yields new invariants at the higher-order homotopy level.

We then introduce the parametric connected-sum construction. This operation takes as inputs two $n$-spheres based at Legendrian embeddings $K_1$ and $K_2$, respectively, and produces a new $n$-sphere based at $K_1\# K_2$. As a main application we construct new infinite families of loops of Legendrian embeddings with non-trivial LCH monodromy invariant.   \end{abstract}

\maketitle
\setcounter{tocdepth}{1} 
\tableofcontents

\section{Introduction}
%\textcolor{red}{Yo pondr\'ia la proposici\'on donde se define la suma conexa en la intro, tambi\'en el teorema principal y los ejemplos del final. IDEAS: This can be thought as a relative version of parametric connected sums of contact structures used in Fernandez-Muñoz to build loops of contact structures blablabla... The operation could be generalized to arbitrary $3$-folds... The parametric connected sum could be useful to build interesting examples of Legendrian loops in the future... We observe that it is possible to realize satellite operations with parameters, something interesting by itself since it gives rise to new invariants of spheres of Legendrians: the satellite sphere...}

\subsection{Context and motivation}. The study of higher homotopy groups of the space of Legendrian embeddings in dimension $3$ has garnered much attention in recent years. 
This article introduces new constructions and invariants at the higher homotopy level for such spaces.
%; namely, the parametric satellite construction and the parametric connected-sum construction.
% This article introduces two constructions, at the higher homotopy level, in the space of Legendrian embeddings: the parametric satellite construction and parametric connected-sums.

The study of Legendrian embeddings has been a central topic in Contact Topology since the work of D. Bennequin \cite{Ben}, where different contact structures on $\R^3$ were distinguished for the first time by means of these objects. Since then, spaces of Legendrian knots have been widely studied at the $\pi_0$-level; see, for instance, \cite{Chekanov, EliashbergFraser, Etnyre, EtnyreHondaTorus, EtnyreHonda, FT, Ng, ost}.

There is, nonetheless, a big lack of understanding of these spaces at the level of higher homotopy groups. T. Kálman laid the groundwork in the study of fundamental groups of such spaces: he constructed the first examples of loops of Legendrian embeddings in $\R^3$ that were trivial as loops of smooth embeddings but non-trivial within the space of Legendrian embeddings \cite{kalman}. He defined a monodromy invariant; i.e. an automorphism of the Legendrian contact homology (see \cite{Chekanov,EliashbergContactInvariants}), which he used to show the non-triviality of such loops. See also \cite{kalman2}.

Motivated by T. Kalman's pioneering work, we continued with the study of higher homotopy elements of the space of Legendrian embeddings in dimension $3$ and started a systematic study of global homotopical properties of Legendrian embedding spaces in contact $3$-manifolds (see \cite{FMP, FMP2, FMP3}).

In \cite{FMP} we proved the non-triviality of the aforementioned loops from \cite{kalman} at the formal level by computing the fundamental group of the space of formal Legendrian embeddings. Furthermore, we constructed infinitely many new examples, by using different techniques from the ones in \cite{kalman, kalman2}, of non-trivial  loops of Legendrians which are trivial as loops of smooth embeddings (see \cite{FMP2}). We also showed that there is a homotopy injection of the contactomorphism group of $\NS^3$ into infinitely many connected components of the space of Legendrians.

In \cite{FMP3} we determined, for the first time, the whole homotopy type of certain connected components of the space of Legendrian embeddings in any given $3$-dimensional tight contact manifold. % for long Legendrian unknots $\gamma$ satisfying the condition $|\rot(\gamma)|+\tb(\gamma)=-1$, where $\rot(\gamma)$ and $\tb(\gamma)$ denote the classical rotation and Thurston-Bennequin invariants, respectively. 
In \cite{FMin}, the first author and H. Min established the first structural results for Legendrian embedding spaces under the cabling operation, a specific instance of the satellite operation studied in this article. Nonetheless, there is still a big lack of understanding of the homotopy type of the whole space of Legendrian embeddings.

This article introduces new tools, which we believe will be fundamental in the study of these spaces at the $\pi_k$-level ($k\geq 1)$, in order to further develop the study of global homotopical properties of these spaces. As a matter of fact, similar constructions and ideas in the smooth category have turned out to be fundamental in establishing classification results for smooth embedding spaces (see \cite{Budney, Budney2, Bud2}). We expect that these constructions will play an analogous role in the contact setting.

As a main application, we construct new infinite families of loops of Legendrian embeddings with non-trivial LCH monodromy invariant. These ideas will play a role as well in upcoming work of the authors.

\subsection{Parametric connected sums and new infinite families of loops.}

Let $(\R^3,\xi_\std=\ker(dz-ydx))$ be the standard contact structure on $\R^3(x,y,z)$. 
Legendrian knots in $(\R^3, \xi_\std)$ are embedded $1$-dimensional submanifolds $L\subseteq \R^3$  everywhere tangent to $\xi_\std$. Similarly, Legendrian embeddings are embeddings parametrizing Legendrian submanifolds. Throughout this article, we will denote by $\widehat{\Leg}$ the space of embedded Legendrians in the standard $(\R^3,\xi_\std)$ and by $\Leg$ the space of Legendrian embeddings. 

One of the main topological constructions of this article (Subsections \ref{ParametricConnectedSums} and \ref{ElefanteMosca}) is the parametric connected-sum operation (Theorem \ref{prop:ConnectedSumMap}) for $k$-spheres of Legendrian embeddings in $(\R^3,\xi_{\std})$. This operation, which can be iterated, takes as inputs two $k$-spheres based at Legendrian embeddings $K_1$ and $K_2$, respectively, and produces a new $k$-sphere based at $K_1\# K_2$.  By iterating it, we can produce, out of any given finite number of embeddings and $k$-spheres based on them, new $k$-spheres of Legendrians based on the connected sum of the given embeddings. 

Subsection \ref{ElefanteMosca} provides a diagrammatic description of the parametric connected-sum. This construction can be generalised to other $3$-manifolds, although its implementation is rather technical (see the PhD thesis \cite{MA} of the second author for its implementation in $(\mathbb{S}^3,\xi_{std})$).

The idea behind this parametric connected-sum construction was inspired to the authors by an argument of J. H. Conway in a totally unrelated context. He used some knotted torus in order to prove a fundamental theorem in Knot Theory  (see \cite[Fig. 31, p. 31]{Gardner}). In the words of M. Gardner, this argument was explained to him ``in a letter'' \cite[p. 60]{Gardner}.

  Let us briefly illustrate J. H. Conway's idea. Take two knots $K_1$ and $K_2$ and realise their connected sum $K_1\# K_2$ as follows: inflate $K_2$ so that it becomes a knotted torus in $\R^3$ and take a knot all along its surface except for some small region where the knot gets inside the knotted torus and realizes the $K_1$ component (dotted trefoil Figure \ref{ConwayIdea}). Note that one can then apply an isotopy to the knotted torus while moving the $K_1$-component rigidly throughout the isotopy (see \cite{Gardner} for further details). We later on realised that similar arguments have been considered in a more general framework within the context of smooth embedding spaces in the work of R. Budney \cite{Bud2, Budney2, Budney}.

\begin{figure}[h]
	\centering
	\includegraphics[width=0.36\textwidth]{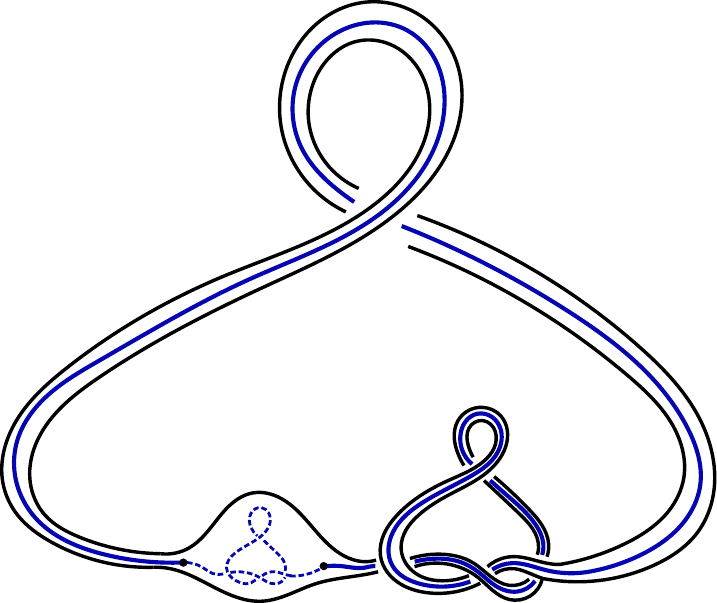}
	\caption{Schematical depiction of J. H. Conway's original idea (originally stated in the context of long knots, although we have adapted to closed knots in this picture) \cite[Fig. 31, p. 61]{Gardner}. \label{ConwayIdea}}
\end{figure}

This beautiful elementary idea, although highly ingenious and elegant, inspired the authors to define parametric connected-sums and satellites in the setting of Legendrian embeddings. 

% Nonetheless, we would like to emphasize that it is philosophically reminiscent of such idea in the sense that it provides, out of given parametric families, a parametric construction where some of the (in this case, Legendrian) components become, in a precise sense, isolated from the topology of the rest of the pieces.

%This parametric connected-sum construction was inspired by a J. H. Conway's argument (in the words of Martin Gardner: \textit{explained in a letter} \cite[p. 60]{Gardner}) where he used some knotted torus in order to prove a fundamental theorem in Knot Theory  (see \cite[Fig. 31, p. 31]{Gardner}). Precisely, he considered the connected sum of two embeddings and carried out an  

  Our construction differs from Conway's, we do not make use of knotted tori and we introduce it in a different framework, although we would like to emphasize that it is reminiscent of that idea. This construction allows us to construct new $n$-spheres of Legendrian embeddings out of given ones (see Theorem \ref{prop:ConnectedSumMap} for the precise statement). In particular, we can apply this parametric connected sum to some of Kálmán's loops \cite{kalman, kalman2} together with infinitely many other loops that we consider, yielding new infinite families of loops of Legendrian embeddings with non-trivial LCH monodromy invariant. Among such loops that we consider, we introduce the ``\textit{Legendrian pulling one knot through another}'' loop for Legendrian embeddings. This is the Legendrian version of the so-called smooth loop described by R. Budney \cite[Sec. 4]{Budney2}.
  
  We prove this last statement by carefully studying the behavior of \kalmans monodromy invariant \cite{kalman} under the parametric connected-sum operation. This requires some non-trivial technical results regarding monodromies of parametric-connected sum loops. Among other things, we study the effects in the LCH monodromy invariant of the changes in size of the embeddings involved during the construction.

    The parametric connected-sum construction that we define combines two different spheres $\gamma_1^k, \gamma_2^k$, $k\in\NS^n$, based at Legendrian embeddings $\gamma_1$, $\gamma_2$, respectively, and generates a new sphere based at $\gamma_1\#\gamma_2$. Roughly speaking, the construction can be described as follows (see Theorem \ref{prop:ConnectedSumMap} for a formal and precise statement): first shrink the $\gamma_1$-component in $\gamma_1\#\gamma_2$ and realise the sphere $\gamma^k_2$ for the second component while moving the $\gamma_1$-component in a rigid way accordingly. Then, inflate the $\gamma_1$-component until it recovers its original size, shrink the $\gamma_2$-component
    and repeat the process for $\gamma^k_2$. Furhtermore we will argue that, for some families, the aforementioned two steps can be carried out simultaneously.
    
  By means of this construction, we provide new families of loops of Legendrians with non-trivial LCH monodromy invariant. These are described in detail in Subsections \ref{familia1}, \ref{familia2} and \ref{ThirdFamily}. Nonetheless let us show, as an appetizer, such an example in Figure \ref{loopsMix} in order to provide the reader with a geometric flavour of the idea behind the construction.

%Roughly speaking, the construction allows to construct new $n$-spheres of Legendrian embeddings out of a given finite number $k>1$ of $n$-spheres of Legendrians Therefore, 

%Roughly speaking, this construction 
%By applying this construction to Kálmán's examples together with infinitely many other loops of Legendrians, we are able to construct

\begin{figure}[h]
	\centering
	\includegraphics[width=0.8\textwidth]{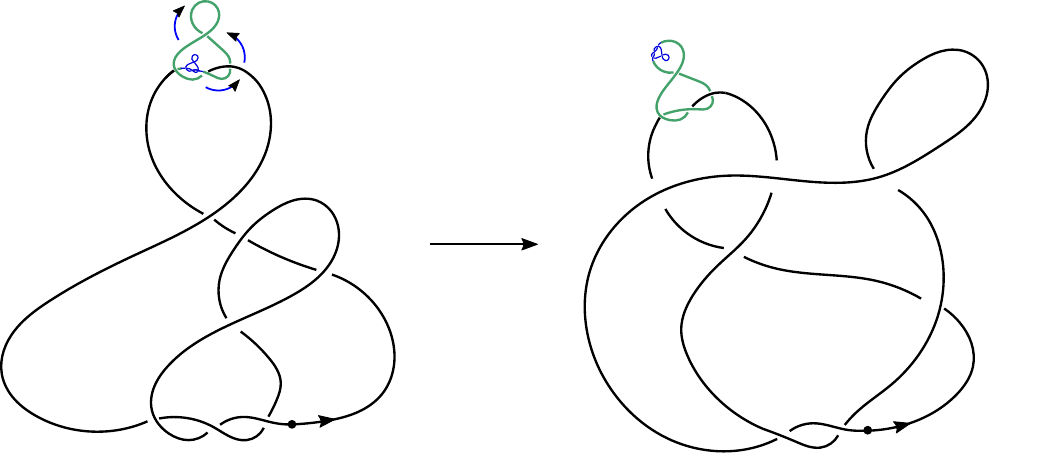}
	\caption{Particular example from the new infinite families of loops with non-trivial monodromy invariant that we construct in Section \ref{argumentomonodromia}. The Figure represents two frames of a loop that belongs to the third family of examples (Subsection \ref{ThirdFamily}) and represents the parametric connected sum of two loops. The big black component is a positive right-handed trefoil performing a Kálmán's loop whereas the blue and green smaller knots are right-handed trefoils simultaneously performing a Legendrian ``pulling one knot through the other'' loop (introduced in Subsection \ref{familia2}).} \label{loopsMix}
\end{figure}

\subsection{Parametric satellite constructions}

The Legendrian satellite construction \cite{EtnyreVertesi, NgTraynor} has turned out to be fundamental, in the sense that it yields structure theorems within the space of Legendrian embeddings for the Legendrian satellite construction. This means that such topological operation allows to determine the structure of some Legendrian embedding isotopy classes in terms of the underlying classes involved in the construction. 

We will show that the classical satellite construction can be performed in families; i.e. works with parameters. The satellite construction can be described as follows. 

Take a Legendrian embedding $\gamma\in \Leg$, called the \textbf{companion embedding}, and fix a contact standard tubular neighborhood of it; i.e. a contact embedding 

\[T_\gamma:(\NS^1\times\D^2,\xi_\std)\rightarrow (\R^3,\xi_\std), \quad\text{where}\quad T_\gamma(t,0,0)=\gamma(t).\]

 Here, $(\NS^1\times\D^2,\xi_\std)\subseteq (J^1\NS^1,\xi_\std)$ is equipped with the contact structure $\xi_\std=\ker(dz-ydt)$ in coordinates $(t,y,z)\in \NS^1\times \R^2=J^1\NS^1$. Take also a Legendrian embedding $l:\NS^1\rightarrow (\NS^1\times\D^2,\xi_\std)$, which will be called the \textbf{pattern embedding}. Out of this data, one can now define the following new embedding.\footnote{The satellite operation can be performed with \em link \em embedding patterns as well. Everything that we say about satellites in this article adapts, word by word, to that case.}

\begin{definition}
    The \textbf{Legendrian satellite} embedding with companion embedding $\gamma$ and pattern embedding $l$ is the Legendrian embedding 
  
    \[ T_\gamma\circ l:\NS^1\rightarrow (\R^3,\xi_\std). \]
    
\end{definition}

A priori, this definition depends on the embedding $T_\gamma$ but we will see below (Lemma \ref{lem:FatLegendrians}) that it is actually independent of this datum up to Legendrian isotopy.

One of the main contributions of this article is the following Theorem, which introduces a parametric Satellite construction for Legendrian embeddings.

\begin{theorem}\label{thm:SatelliteMap}
    Let $l:\NS^1\rightarrow (\NS^1\times\D^2,\xi_\std)$ be a Legendrian embedding. There is a well defined continuous Legendrian satellite map 
    \begin{equation}\label{eq:SatelliteMap} \Sat(l):\Leg\rightarrow \Leg,\quad \gamma\mapsto \Sat(l)(\gamma),
    \end{equation}
satisfying the following property. For every compact family of Legendrians $\gamma^k\in \Leg$, $k\in K$, its image $\Sat(l)(\gamma^k)$, $k\in K$, is homotopic, through families of Legendrian embeddings, to a family which is obtained by applying simultaneously the satellite construction to every $\gamma^k\in \Leg$, $k\in K$. 
\end{theorem}

As a consequence, we can define new invariants at the higher-homotopy level arising from this construction as follows.

\begin{corollary}\label{SatelliteInvariants}
        For every $n>0$ and any Legendrian embedding $l:\NS^1\rightarrow (\NS^1\times\D^2,\xi_\std)$, the parametric Legendrian satellite map induces well defined group homomorphisms at the $\pi_n$-homotopy level:
        \[ \Sat(l)_{*n}:\pi_n\left(\Leg,\gamma\right) \rightarrow \pi_n\left(\Leg,\Sat(l)(\gamma)\right). \]
    \end{corollary}

%As an application of this construction, we obtain the following Structure theorem for $\pi_n$-surjectivity.

\textbf{Acknowledgements:} The first author wants to thank Wei Zhou for valuable discussions. The second author would like to thank Tobias Ekholm for kindly answering his questions and for the nice Remark \ref{EkholmRemark}. This gave us the courage to prove that the constructed loops have indeed non-trivial LCH monodromy invariant. He would also like to thank Fabio Gironella for his suggestions and fine comments to a preliminary version of his PhD thesis and Marián Poppr for useful conversations. Finally, he would like to thank his advisor Álvaro del Pino and the Geometry group of Utrecht University where he delivered a talk about parametric connected-sums in 2019 yielding valuable discussions, as well as Antonio Alarcón and the Institute of Mathematics of the University of Granada (IMAG) for providing him with nice environments to develop this and other projects. 

During part of the development of this work the first author was supported by ``Beca de Personal Investigador en Formación UCM'' and the second author was supported by the ``Programa Predoctoral de Formación de Personal Investigador No Doctor'' scheme funded by the Basque Department of Education (``Departamento de Educación del Gobierno Vasco''). The authors were also funded by the
Spanish Research Project MTM2016-79400-P.

\section{Constructions in the space of Legendrian embeddings}

In this Section we introduce the two main geometric constructions in the article; namely, the parametric connected-sum and satellite constructions. First, we will see that the satellite operation can be carried out parametrically by making use of what we will call \textit{fat Legendrians}. 

We will later on move to the second main construction: parametric connected sums. We will interpret the (usual) Legendrian connected sum operation of two Legendrians introduced in \cite{EtnyreHonda} by means of \textit{Legendrians in standard position}, which are a special type of satellites of the standard Legendrian unknot. This will allow us to define a parametric notion of connected sums. Finally, we will explain how to diagrammatically depict this operation by means of what we call the \textit{Elephant-Fly} construction.% and generalize it for families of Legendrians by making use of fat Legendrians. We will view the Legendrian First, we introduce the notions of Legendrians in standard positions which are Legendrian satellites of a standard Legendrian unknot \cite{EtnyreVertesi,NgTraynor}; and fat Legendrians which we will use to produces satellites of families of Legendrians. Then, we will view the connected sum of two Legendrians introduced in \cite{EtnyreHonda} as a Legendrian in standard position and generalize it for families of Legendrians by making use of fat Legendrians. Finally, we will give a diagrammatic description of the parametric connected sum in the Lagrangian projection that we name Elephant-Fly construction.

\subsection{Legendrian satellite embeddings.}

We will denote by $\D^k\subseteq \R^k$  the euclidean unit $k$-disk centered at the origin. For the sake of readability, we make the non-standard identification $\NS^1=[-2,2]/\simeq$.

\begin{definition}
    A \textbf{fat Legendrian} embedding is a contact embedding $T:(\NS^1\times\D^2,\xi_\std)\rightarrow (\R^3,\xi_\std)$.
\end{definition}

Denote the space of fat Legendrians as $\FatLeg$. Observe that given any pattern $l:\NS^1\rightarrow (\NS^1\times\D^2,\xi_\std)$ there is a well defined Serre fibration

\begin{equation}\label{eq:SatelliteFat}
G_l:\FatLeg\rightarrow \Leg, \quad T\mapsto T\circ l.
\end{equation}

Note that by taking a companion embedding $\gamma$  and fixing a contact tubular neighborhood $T_\gamma$ of it as before, $G_l(T_\gamma)$ is thus a Legendrian satellite with companion $\gamma$ and pattern $l$. A particular pattern of interest is the $0$-section $l_0:\NS^1\hookrightarrow \NS^1\times\{0\}\subseteq (\NS^1\times \D^2,\xi_\std)$ with associated fibration

\begin{equation}\label{eq:FatRestriction}
 G_0:=G_{l_0}:\FatLeg\rightarrow \Leg, \quad T\mapsto T\circ l_0.
\end{equation}

The fiber of $G_0$ over $\gamma$ is $\FatLeg_{\gamma}=\{T\in\FatLeg :\gamma=T\circ l_0\}$; i.e. the space of contact tubular neighborhoods of $\gamma$.

\begin{proof}[Proof of Theorem \ref{thm:SatelliteMap}]
The result is a consequence of the following Lemma.%\noindent\textit{Proof of Theorem \ref{thm:SatelliteMap}}. The result is a consequence of the following Lemma.

\begin{lemma}\label{lem:FatLegendrians}
     The fibers $\FatLeg_{\gamma}$ of the fibration $G_0:\FatLeg\rightarrow \Leg$ are contractible and, thus, $G_0$ is a homotopy equivalence.
\end{lemma}

Indeed, after the Lemma, we may consider a homotopy inverse $\Fat:\Leg\rightarrow \FatLeg$ of $G_0$ and define $\Sat(l)(\cdot):=G_l\circ \Fat(\cdot)$.
\end{proof}

\begin{proof}[Proof of Lemma \ref{lem:FatLegendrians}]
     We will see that $\FatLeg_{\gamma}$ is weakly contractible\footnote{The Whitehead Theorem applies to all the spaces considered in this article.}. Let $T^k\in\FatLeg_\gamma$, $k\in K$, be a compact family of fat Legendrians. Fix a standard neighborhood $(J^1\NS^1,\xi_\std)$ of $\gamma(\NS^1)$, in such a way that 
     
     $$\gamma(t)=(t,0,0)\in (J^1\NS^1,\xi_\std)\subseteq (\R^3,\xi_\std).$$ 
     
     Notice that, by definition, $T^k(t,0,0)=\gamma(t)=(t,0,0)\in (J^1\NS^1,\xi_\std)\subseteq (\R^3,\xi_\std)$. Therefore, the existence of contact contractions along the fibers $\phi^u(t,y,z)=(t,uy,uz)$, $u>0$, in $(J^1\NS^1,\xi_\std)$, together with the compactness of the family $T^k$, allows us to assume that the image of each element in the family lies within $(J^1\NS^1,\xi_\std)$; i.e. we can write
     
     \[T^k:(\NS^1\times\D^2,\xi_\std)\rightarrow (J^1\NS^1,\xi_\std)\subseteq (\R^3,\xi_\std).\]

   It suffices to find a homotopy of contact embeddings $T^{k,u}:(\NS^1\times \D^2,\xi_\std)\rightarrow (J^1\NS^1,\xi_\std)\subseteq (\R^3,\xi_\std)$, $(k,u)\in K\times[-1,1]$, such that 
   \begin{itemize}
       \item [(i)] $T^{k,u}\in \FatLeg_\gamma$, for $(k,u)\in K\times[-1,1]$;
       \item [(ii)] $T^{k,1}=T^k$, for $k\in K$; and 
       \item [(iii)] $T^{k,-1}=i:(\NS^1\times\D^2,\xi_\std)\hookrightarrow (J^1\NS^1,\xi_\std)\subseteq (\R^3,\xi_\std)$ is the inclusion for every $k\in K$. 
   \end{itemize}
   
   Notice that condition (i) just means that $T^{k,u}(t,0,0)=\gamma(t)=(t,0,0)\in (J^1\NS^1,\xi_\std)\subseteq (\R^3,\xi_\std)$; i.e. $T^{k,u}$ preserves the $0$-section. Define the homotopy over $K\times[0,1]$ as the family of contact embeddings
   
   $$T^{k,u}=\phi^{\frac{1}{u}}\circ T^k\circ \phi^u,\quad (k,u)\in K\times[0,1].$$
   
   Here $T^{k,0}(t,y,z)$ is defined as the limit $T^{k,0}(t,y,z):=\lim_{u\mapsto 0} \phi^{\frac{1}{u}}\circ T^k\circ \phi^u(t,y,z)$. Note that $T^{k,0}$ is the ``vertical differential'' of $T^k$ along the $0$-section, so it is linear in the fiber directions. The contact embedding condition implies that
   
   $$T^{k,0}(t,y,z)=(t,\lambda^k(t)\cdot y +(\lambda^k)'(t)\cdot z,\lambda^k(t) \cdot z)$$ 
   
   for some positive function $\lambda^k:\NS^1\rightarrow (0,\infty)$. Extend the homotopy over $K\times[-1,0]$ by 
   
    $$ T^{k,u}(t,y,z)=(t, \lambda^{k,u}(t)\cdot y + (\lambda^{k,u})'(t)\cdot z, \lambda^{k,u} (t) \cdot z),$$
    
    where $\lambda^{k,u}=(1+u)\lambda^k-u$. This concludes the proof.
\end{proof}

\subsection{Legendrians in standard position and Legendrian tangles}\label{LongKnotsTangles}

\begin{definition}
A long Legendrian embedding $\hat{\gamma}:\R\to(\R^3,\xi_\std)$ is a Legendrian embedding such that $\hat{\gamma}(t)=(t, 0, 0)$ for $|t|>1$ and $\hat{\gamma}(\R)\cap\D^3=\hat{\gamma}[-1,1]$.
\end{definition}

Associated to a long Legendrian embedding $\hat{\gamma}$ there is a Legendrian arc $\tilde{\gamma}=\hat{\gamma}_{|[-1,1]}:[-1,1]\rightarrow (\D^3,\xi_\std)$. We will also refer to the Legendrian arc $\tilde{\gamma}$ as a \textbf{Legendrian tangle}. Note that all the non-trivial part of a long embedding $\hat{\gamma}$ may be shrunk to take place near the origin. Indeed, the diffeomorphism $\varphi^u(x,y,z)=(ux,uy,u^2z)$ is a contactomorphism for every $u>0$. Therefore, $\gamma^u(t):=\varphi^u\circ \gamma(\frac{1}{u} t)$ is a long Legendrian embedding for every $u\in(0,1]$ for which the non-trivial part of the embedding takes place in the $3$-disk of radius $u$. This gives raise to an isotopy of Legendrian tangles $\tilde{\gamma}^u$ as follows.

\begin{definition}\label{shrinking}
    Let $u\in(0,1]$. The \textbf{$u$-shrinking} of $\tilde{\gamma}$ is defined as the Legendrian tangle $S^u(\tilde{\gamma}):=\tilde{\gamma}^u$.
\end{definition}

Now we explain the closure operation for Legendrian tangles as a Legendrian satellite operation. Let $\beta:\NS^1\to(\R^3,\xi_\std)$ be the \textbf{standard Legendrian unknot} in $\left(\R^3,\xi_{\std}\right)$ which we assume to satisfy $\beta(t)=(t,0,0)$ for $t\in[-1,1]\subseteq \NS^1$. 

 For every Legendrian tangle $\tilde{\gamma}$ there is an associated Legendrian pattern $l_{\tilde{\gamma}}:\NS^1\rightarrow (\NS^1\times\D^2,\xi_\std)$, namely,
 \[
l_{\tilde{\gamma}}(t) = \begin{cases}
(t,0,0) & \text{if } t \in \NS^1 \setminus [-1,1], \\
\tilde{\gamma}(t) & \text{if } t \in [-1,1].
\end{cases}
\]

Note that for for $t\in\NS^1\backslash [-1,1]$, the curve is just the $0$-section $l_{\tilde{\gamma}}(t)=l_0(t)=(t,0,0)$. 

\begin{definition}\label{Closure} The \textbf{Legendrian closure} of $\tilde{\gamma}$ is the Legendrian embedding \[\gamma=\Sat(l_{\tilde{\gamma}})(\beta).\]  We will refer to $\tilde{\gamma}$ as the \textbf{$\gamma$-tangle}.
\end{definition}

A Legendrian embedding obtained as the closure of a tangle is said to be in \textbf{standard position}. The operation is depicted in Figure \ref{LongToShort} and is also explained, from equivalent approaches, in \cite[3.5]{FT} and \cite{DingGeiges}. Furthermore, F. Ding and H. Geiges showed in \cite{DingGeiges} that isotopy classes of long Legendrian embeddings/Legendrian tangles are in one-to-one correspondence, via the aforementioned process, with isotopy classes of Legendrian embeddings. 

\begin{figure}[h]
	\centering
	\includegraphics[width=0.6\textwidth]{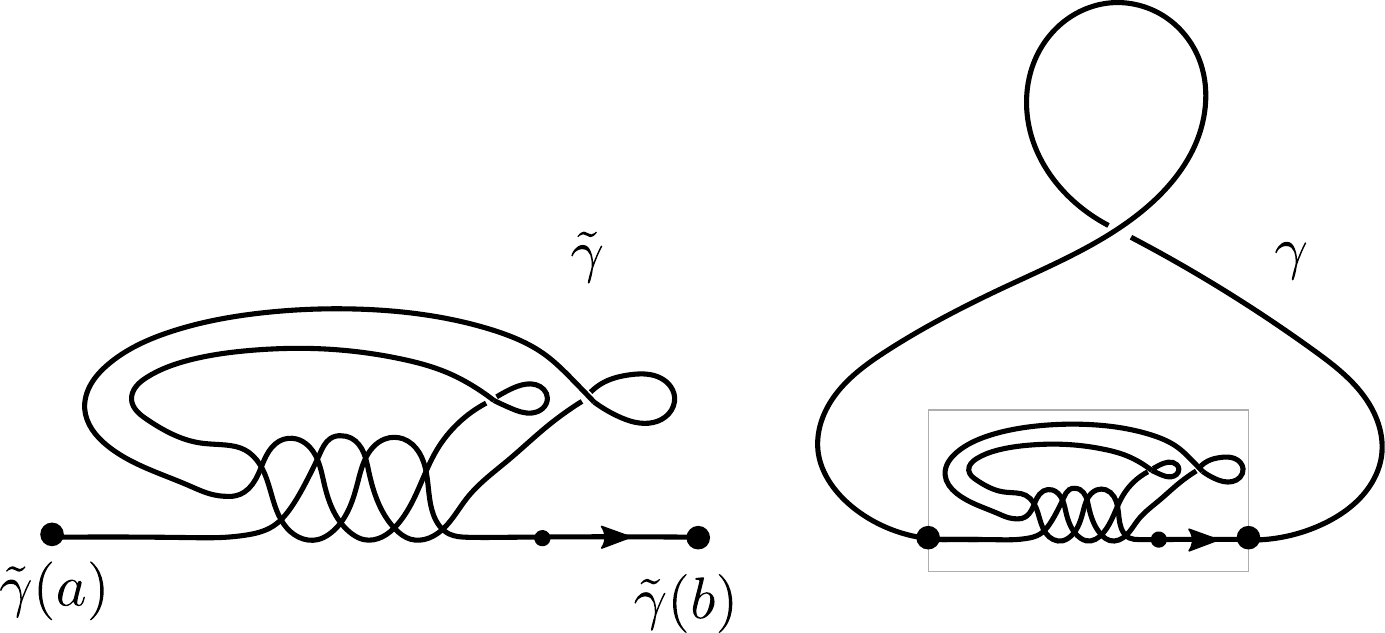}
	\caption{On the left, a Legendrian tangle $\tilde{\gamma}$ corresponding to a $K_{4,3}$ torus long knot. On the right, its Legendrian closure $\gamma$. Equivalently, we say that the Legendrian $K_{4,3}$ torus knot on the right is in standard position. \label{LongToShort}}
\end{figure}

\subsection{Legendrian connected sum}\label{ConnectedSums}

The Legendrian connected sum \cite{EtnyreHonda} of two Legendrian embeddings $\gamma_1, \gamma_2$ is a well-defined operation \cite{EtnyreHonda} and it admits the following diagrammatic description in $(\R^3, \xi_{\std})$  (see \cite[5.4]{Etnyre}). Working in the Lagrangian projection, we can assume without loss of generality (up to a translation) that one of the embeddings is on top of the other. Choose a small arc $\gamma_1(-\varepsilon, \varepsilon)$ for $\gamma_1$ and a contiguous small arc $\gamma_2(-\varepsilon, \varepsilon)$ for $\gamma_2$.

\begin{definition}\label{ConnectedSum}
The Legendrian \textbf{connected sum} $\gamma_1\#\gamma_2$ is the Legendrian embedding whose Lagrangian projection is obtained by joining $\gamma_1\left(\NS^1\setminus(-\varepsilon, \varepsilon)\right)$ with $\gamma_2\left(\NS^1\setminus(-\varepsilon, \varepsilon)\right)$ via the local model shown in Figure \ref{LagrangianSum} and whose orientation is also determined by the local picture in Figure \ref{LagrangianSum}.
\end{definition}

\begin{figure}[h]
	\centering
	\includegraphics[width=0.4\textwidth]{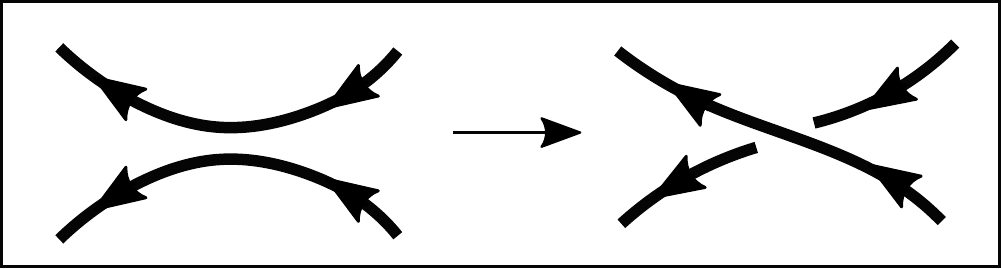}
	\caption{Diagrammatic local model for the connected sum of two Legendrian embeddings in the Lagrangian projection. Note that the local model for the arcs with reversed orientations follows from a $\pi$-radian rotation of this local model in the plane. \label{LagrangianSum}}
\end{figure}

Consider $k$ Legendrian embeddings $\gamma_1, \cdots, \gamma_k$ in $\left(\R^3,\xi_{\std}\right)$ that we assume to be in standard position.

\begin{proposition}\label{StandardModelPosition}
 $\gamma_1\#\cdots\#\gamma_k$ admits a representation in standard position with associated tangle given by the ordered concatenation of the $\gamma_i$-tangles, respectively, for $i=1,\ldots, k$. 
\end{proposition}
\begin{remark}\label{rmk:Reparametrizations}
    The concatenation of the $\gamma_i$-tangles must be reparametrized in order to have domain $[-1,1]$. For instance, in the case of two Legendrians $\gamma_1$ and $\gamma_2$, then (recall Definition \ref{shrinking}) the associated $\gamma_1\#\gamma_2$-tangle is defined as 
    \[
\widetilde{\gamma_{1}\#\gamma_{2}}(t) = \begin{cases}
S^{\frac{1}{2}}(\tilde{\gamma}_{1})(t+\frac{1}{2}) & \text{if } t\in[-1,0], \\
S^{\frac{1}{2}}(\tilde{\gamma}_{2})(t-\frac{1}{2}) & \text{if } t \in [0,1].
\end{cases}
\]

\end{remark}
\begin{proof}[Proof of Proposition \ref{StandardModelPosition}]
It is enough to show that this result holds for the connected sum of two knots since the argument can be iterated. The result follows by applying a Reidemeister $II^{-1}$ move to the connected sum as depicted in Figure \ref{connectedTangles}. Note that the area of the $2-$gon in dark grey can be made smaller than the sum of the areas of adjacent regions (in light grey) at its vertices (for example, by taking the $\gamma_2$ very small from the beginning). Since this is the condition in order for this Reidemeister $II^{-1}$-move to be valid in the Legendrian category (see \cite[Theorem 4.1]{kalman}), the result follows.
\end{proof}

\begin{figure}[h]
	\centering
	\includegraphics[width=0.7\textwidth]{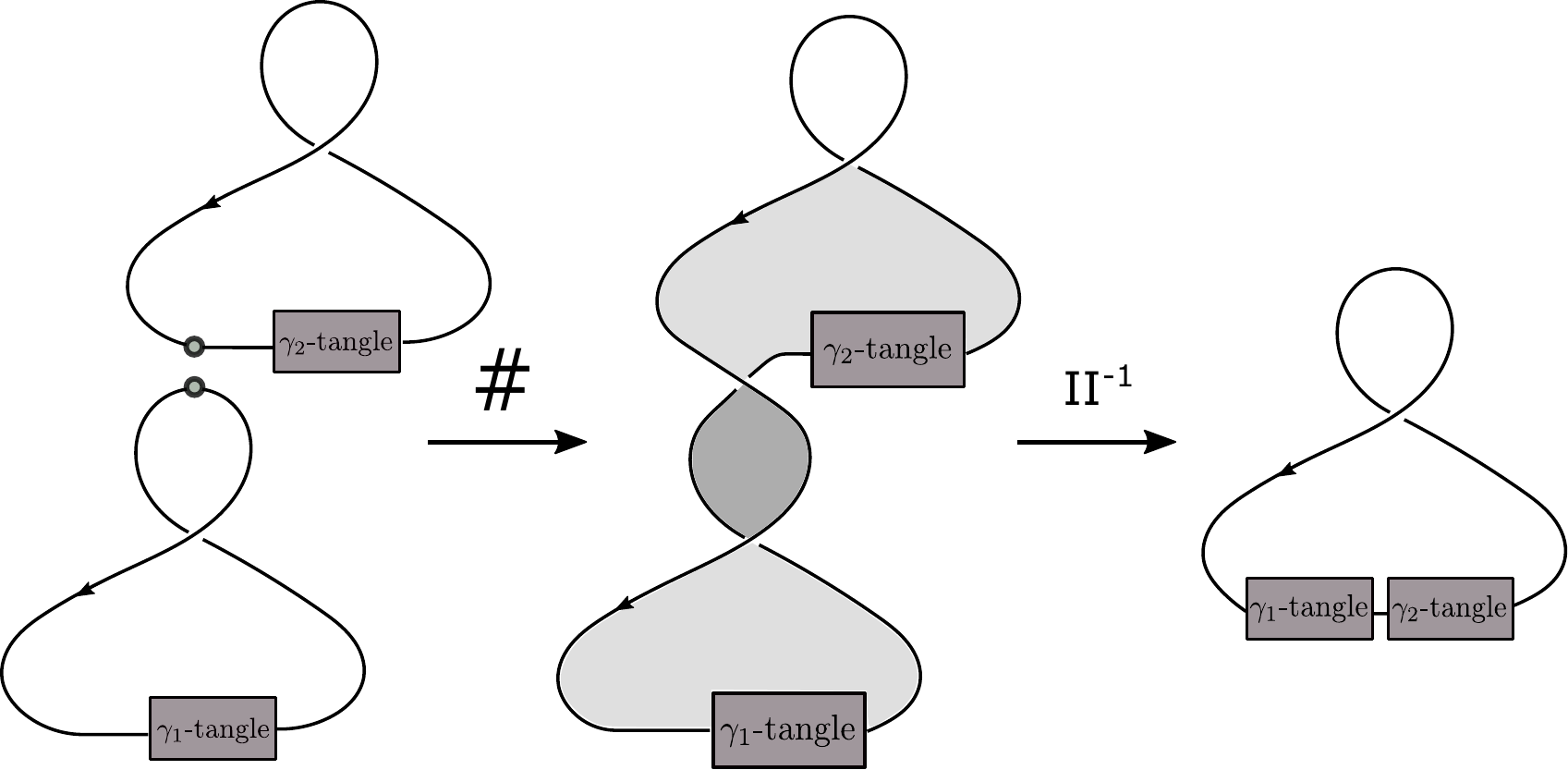}
	\caption{Legendrian connected sum of two embeddings $\gamma_1, \gamma_2$ followed by a legit Reidemeister $II^{-1}$ move. The area of the involved $2-$gon (in dark grey) can be made smaller than the sum of the areas of adjacent regions (in light grey) at its vertices.\label{connectedTangles}}
\end{figure}

    Let $\gamma_1\#\gamma_2$ be the connected sum, in standard position, of the two Legendrians $\gamma_1$ and $\gamma_2$ that are also in standard position obtained by means of Proposition \ref{StandardModelPosition}. Then,
     \[  \gamma_1\#\gamma_2=\Sat(l_{\widetilde{\gamma_1\#\gamma_2}})(\beta). \]
    Recall that $\beta$ stands for the \textbf{standard Legendrian unknot} in $\left(\R^3,\xi_{\std}\right)$ as in Definition \ref{Closure}.
    Following the notation from Remark \ref{rmk:Reparametrizations} we consider tangles $\tilde{\gamma}_1=\widetilde{\gamma_1\#\beta}$ and  $\tilde{\gamma}_2=\widetilde{\beta\#\gamma_2}$. Do note that, $\Sat(l_{\widetilde{\gamma_1\#\beta}})(\beta)$ coincides with $\beta$ over $\NS^1\backslash[-1,0]$ and $\Sat(l_{\widetilde{\gamma_1\#\beta}})(\beta)$ coincides with $\beta$ over $\NS^1\backslash[0,1]$. Moreover, both Legendrians are isotopic in a trivial way to $\gamma_1$ and $\gamma_2$, respectively. We will make an abuse of notation and just name them $\gamma_1$ and $\gamma_2$. %\textcolor{red}{Figura: En la figura 4 el ultimo nudo sin el segundo tangle es $\gamma_1$ y sin el primero es $\gamma_2$} \textcolor{blue}{Xabi: sí, estoy de acuerdo. Pero qué quieres decir con esto? Que añadamos ese remark? }
    
    With this notation, observe that the following symmetric relation holds:
    \begin{equation}\label{eq:Symmetry}
    \gamma_1\#\gamma_2=\Sat(l_{\tilde{\gamma}_2})(\gamma_1)=\Sat(l_{\tilde{\gamma_1}})(\gamma_2).
    \end{equation}
    This is clear from the diagram (see, for instance, the last picture on the right in Figure \ref{connectedTangles}). Nonetheless, it can also be formalized by taking tubular neighborhoods of $\gamma_1$ (respectively, $\gamma_2$) that coincide with the tubular neighborhood of $\beta$ over $\NS^1\backslash[-1,0]$ (respectively, $\NS^1\backslash[0,1]$) to realize the satellites. Lemma \ref{lem:FatLegendrians} ensures that we can do this. 
    \subsection{Parametric connected sums}\label{ParametricConnectedSums}
    It follows from Theorem \ref{thm:SatelliteMap} the existence of well defined continuous maps 
    \[ \Phi_1=\Sat(l_{\tilde{\gamma}_1}):(\Leg,\gamma_2)\rightarrow (\Leg,\gamma_1\#\gamma_2) \]
    and 
    \[ \Phi_2=\Sat(l_{\tilde{\gamma}_2}):(\Leg,\gamma_1)\rightarrow (\Leg,\gamma_1\#\gamma_2)\]
    
    between \em pointed \em spaces, where the image $\Phi_1(g)$ of an embedding $g\in\Leg$ is an embedding obtained from $g$ by performing connected sum with $\gamma_1$ (the symmetric statement holds by exchanging the subindices $1$ and $2$). Likewise, this construction works with parameters: for every $n>0$ these maps induce well defined continuous maps between the $n$-fold based loop spaces 
    
    \[ \Omega^n\Phi_1:\Omega^n_{\gamma_2}\Leg\rightarrow \Omega^n_{\gamma_1\#\gamma_2} \Leg \] and \[ \Omega^n\Phi_2:\Omega^n_{\gamma_1}\Leg\rightarrow \Omega^n_{\gamma_1\#\gamma_2} \Leg, \]

   where $\Omega^n\Phi_1$ maps each sphere $f\in\Omega^n_{\gamma_2}\Leg$ to a new sphere $\Phi_1(f)\in\Omega^n_{\gamma_1\#\gamma_2}\Leg$ which is homotopic, through spheres in $\Omega^n_{\gamma_1\#\gamma_2}\Leg$, to a family obtained from $f$ by performing connected sum with $\gamma_1$. Once again, the symmetric statement holds by exchanging the subindices $1$ and $2$.

     In particular, given $f_1\in\Omega^n_{\gamma_1} \Leg$ and $f_2\in\Omega^n_{\gamma_2} \Leg$ there is an associated concatenation $n$-sphere
     
     \[ \Phi_1(f_2)\cdot\Phi_2(f_1)\in \Omega^n_{\gamma_1\#\gamma_2}\Leg,\] 
     
     where $\cdot$ stands for the concatenation of based spheres. We encapsulate this result in the following Theorem, which represents one of the main constructions introduced in this article.

    \begin{theorem}\label{prop:ConnectedSumMap}
        For every $n>0$ there exists a well defined continuous Legendrian \textbf{parametric connected sum} map 
        
         \[ \#:\Omega^n_{\gamma_1}\Leg\times\Omega^n_{\gamma_2}\Leg \rightarrow \Omega^n_{\gamma_1\#\gamma_2} \Leg, \]
        
        which assigns to each pair $(f_1, f_2)\in \Omega^n_{\gamma_1}\Leg\times\Omega^n_{\gamma_2}\Leg$ the concatenation  $\Phi_1(f_2)\cdot\Phi_2(f_1)\in \Omega^n_{\gamma_1\#\gamma_2}\Leg$.
    \end{theorem}
 Given two $n$-spheres $f_i(k)=\gamma_i(t,k)\in\Leg$, $(k,i)\in\NS^n\times\{1,2\}$, based at $f_i(N)=\gamma_i(t,N)=\gamma_i(t)$, we will denote the sphere obtained by means of the previous map $\#$ by \[\gamma_1\#\gamma_2(t,k) \]
 
 and we will call it the \textbf{parametric connected sum} of $\gamma_1(t,k)$ and  $\gamma_2(t,k)$.

    The following Subsection is devoted to diagrammatically describe the connected-sum map. This will be key in order to construct the new examples we present in Subsection \ref{LoopsExamples}.

\subsection{Diagrammatic representation of the parametric connected sum: the Elephant-Fly construction.}\label{ElefanteMosca}

In order to provide practical applications, construct new examples and work with the connected sum map from Theorem \ref{prop:ConnectedSumMap}, it will be useful to have in mind a diagrammatic representation of this operation. This will be explained in the present Subsection.

Let $\gamma_1\#\gamma_2(t,k)$ be the connected-sum $n$-sphere obtained out of two given $n$-spheres $\gamma_1(t,k)$, $\gamma_2(t,k)$ by means of Theorem \ref{prop:ConnectedSumMap}. In this Section we explain how to diagrammatically represent the Lagrangian projection of (a homotopic) sphere of Legendrians. We call this representation the \textbf{Elephant-Fly construction}. This terminology will be justified along the description. 

Note that by definition $\gamma_1\#\gamma_2(t,k)$ is the concatenation of two $n$-spheres $\mathcal{E}_1(t,k):=\Phi_2(\gamma_1(t,k))$ and $\mathcal{E}_2(t,k):=\Phi_1(\gamma_2(t,k))$. We will homotope both spheres $\mathcal{E}_1$ and $\mathcal{E}_2$, respectively, into new spheres for which the Lagrangian projection will be accurately described. We explain the procedure for $\mathcal{E}_1$, since for $\mathcal{E}_2$ is completely analogous. 

First, up to possibly having to perform an initial homotopy, we can assume that $\gamma_1(t,k)=\gamma_1(t)$ for every $k\in U\subseteq \NS^n$, where $U$ is some small open disk centered at the north pole $N\in\NS^n$. By definition, \[  \mathcal{E}_1(t,k)=\Phi_2(\gamma_1(t,k))=\Sat(l_{\tilde{\gamma}_2})(\gamma_1(t,k)). \]

We will perform a first deformation of $\mathcal{E}_1(t,k)$, which will continuously shrink the $\gamma_2$-tangle radially along the disk $U$. Roughly speaking, this will produce the effect of somehow ``isolating'' the topology of the $\gamma_2$-tangle.

In order to do so, fix the height function $h:\NS^n\rightarrow [-1,1]$ on the sphere, a small enough $\delta>0$ and a smooth function $\rho=\rho_\delta:\NS^n\rightarrow[\delta,1]$ such that:
\begin{itemize}
    \item[i)] $\rho|_{\NS^n\backslash U}\equiv \delta$,
    \item[ii)] $\rho(N)=1$ and
    \item[iii)] $d\rho(\nabla h)\geq0$.
\end{itemize}  If $\delta>0$ is chosen small enough, the family of arcs $\gamma_1(t,k)|_{[\frac{1}{2}-\delta,\frac{1}{2}+\delta]}$ can be regarded as ``almost constant'', i.e.  they look like a small segment for all $k\in K$. Use $\rho$ to define the family of tangles 

\[ \tilde{\gamma}_2(t,(k,u)):= S^{(1-u)\rho(k)+u}(\tilde{\gamma}_2)(t),\quad (k,u)\in\NS^n\times[0,1]. \]

Note that $\tilde{\gamma}_2(t,(k,1))= S^1(\tilde{\gamma}_2)(t)=\tilde{\gamma}_2(t)$ for all $k\in \NS^n$ and that $\tilde{\gamma}_2(t,(k,0))=S^\delta(\tilde{\gamma}_2)(t)$, for $k\in \NS^n\backslash U$.

This family of tangles allows to define the homotopy of spheres 

\[\mathcal{E}_1(t,k,u)=\Sat( l_{\tilde{\gamma}_2(t,(k,u))})(\gamma_1(t,k)) \]

between $\mathcal{E}_1(t,k)=\mathcal{E}_1(t,k,1)$ and the sphere $\mathcal{E}_1(t,k,0)$.

The desired sphere of embeddings is $\mathcal{E}_1(t,k,0)$, that in an abuse of notation we just relabel as $\mathcal{E}_1(t,k)$. Depicting the Lagrangian projection of this construction is straightforward after these considerations:
\begin{itemize}
    \item [(i)] \textbf{Shrinking:} For $k\in U$, the Lagrangian projection of $\mathcal{E}_1(t,k)$ coincides with the Lagrangian projection of $\gamma_1\#\gamma_2$ except in the $\gamma_2$-tangle region, in which there is a continuous shrinking of the $\gamma_2$-tangle, along the flow lines of $-\nabla h$, until it becomes arbitrarily small; i.e. it becomes $S^\delta(\tilde{\gamma}_2)$ for a small enough $\delta>0$. See Figure \ref{InflatingIsotopyConnectedSum}.

    \begin{figure}[h]
	\centering
	\includegraphics[width=0.65\textwidth]{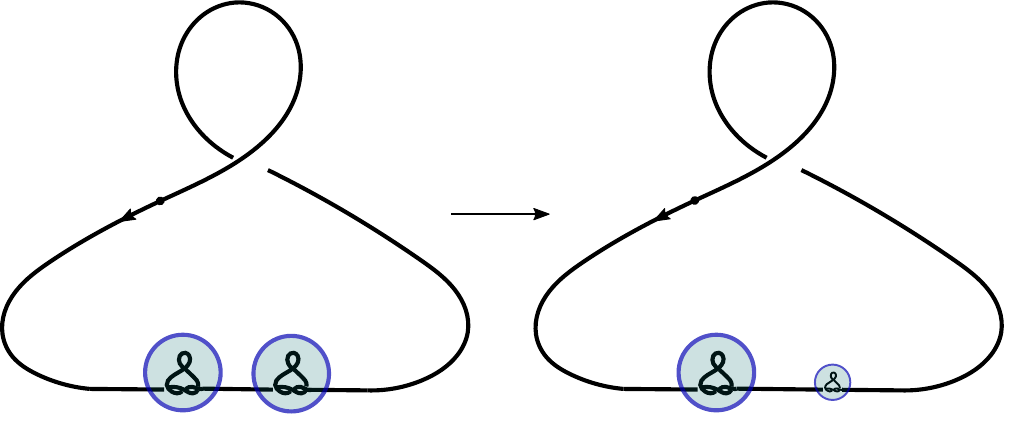}
	\caption{Schematic depiction of the shrinking of the $\gamma_2$-tangle (right) until it becomes arbitrarily small. After this shrinking process, the $\gamma_2$-tangle can be represented as a small tangle (the fly) moving rigidly along the $n$-sphere. Actually, we may even represent it as a small box later on (See Figure \ref{Fig:LoopBoxRotation}, for instance). Do note that the rest of the embedding (corresponding to the exterior of the $\gamma_2$-region) remains unaffected by this prior deformation. }\label{InflatingIsotopyConnectedSum}
\end{figure}

    \item[(ii)] \textbf{The fly along the elephant:} For $k\in \NS^n\backslash U$, the Lagrangian projection of $\mathcal{E}_1(t,k)$ coincides with the one of $\gamma_1(t,k)$ except at the very small segment $\gamma_1(t,k)_{|[\frac{1}{2}-\delta,\frac{1}{2}+\delta]}$ that is replaced by the $\delta$-shrinking tangle $S^\delta(\tilde{\gamma}_2)$. The value $\delta>0$ is fixed to be arbitrarily small so that the (very small) tangle $S^\delta(\tilde{\gamma}_2)$ (the fly) can be understood as just moving rigidly along the family $\gamma_1(t,k)$ (the elephant). See Figure \ref{Fig:LoopBoxRotation}.

\begin{figure}[h]
	\centering
	\includegraphics[width=0.75\textwidth]{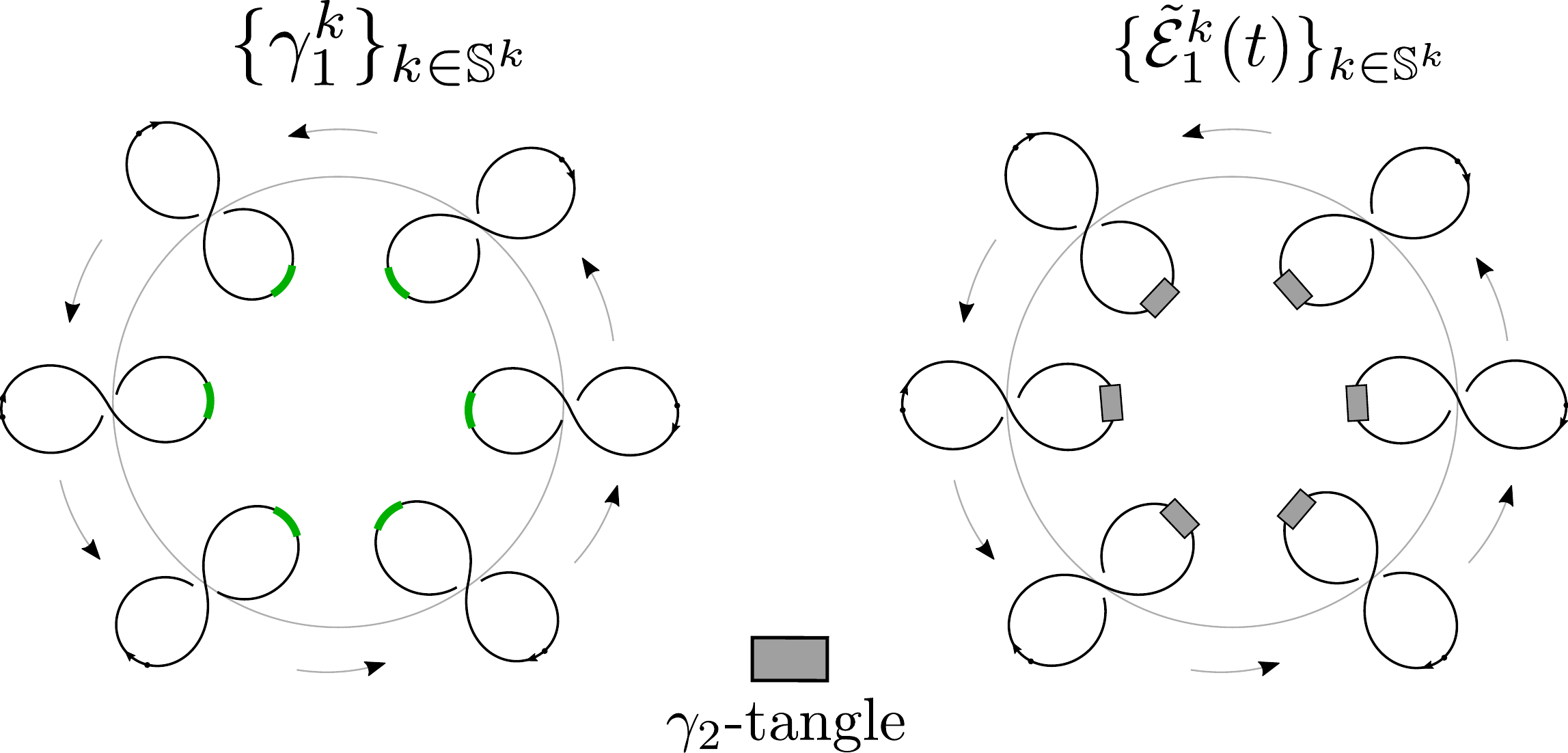}
	\caption{The original $k$-sphere $\lbrace\gamma^k_1\rbrace_{k\in\NS^n}$ of Legendrian embeddings is depicted on the left. In this case we have depicted a $1$-sphere corresponding to a Lagrangian rotation of a Legendrian unknot. On the right we depict the $n$-sphere  $\lbrace{{\mathcal{E}}}^k_1(t)\rbrace_{k\in\NS^n}$ (right) obtained by replacing a small segment (marked in green on the left) by an arbitrarily small $\gamma_2$-tangle (the fly). The new $n$-sphere can thus be regarded as the original embedding (the elephant) realising the original $n$-sphere with a small box/tangle attached to it (the fly) and moving rigidly with it.  \label{Fig:LoopBoxRotation}}
\end{figure}

\end{itemize}

After realising the sphere $\mathcal{E}_1$, since the whole sphere is defined as the concatenation of $\mathcal{E}_1$ with $\mathcal{E}_2$, then one proceeds to unshrink/inflate the $\gamma_1$-tangle (fly) until it recovers its original size. Then, the sphere $\mathcal{E}_2$ is realised by repeating the process but with the roles of the $\gamma_1$-tangle and the $\gamma_2$-tangle reversed (i.e. now the $\gamma_2$-tangle plays the role of the fly and viceversa).

\begin{remark}
    The terminology \textit{\textbf{Elephant-fly construction}} is motivated by (ii) since it resembles a moving elephant (the embedding $\gamma_1$) with a fly attached to its back (the $\delta$-shrinking of the $\gamma_2$-tangle). The calligraphic ``E'' letter $\mathcal{E}_1$ in the notation encodes that the embedding $\gamma_1$ takes the role of the elephant. We will likewise say that $\gamma_2$ takes the role of the fly.
\end{remark}

We will finally discuss some particular instances of $n$-spheres of Legendrians; namely  $n$-spheres which are the closure of an $n$-sphere of tangles, which admit very easily representable depictions. 

\subsubsection{Parametric connected sums and spheres of embeddings in standard position}

A particular type of $n$-spheres of embeddings which can be of special interest are those that take place ``within a box'' or, more precisely, those that consist of Legendrian embeddings which are in standard position. Let us elaborate on this.

Consider an $n$-sphere
$\lbrace\hat{\gamma}^k\rbrace_{k\in\NS^n}$ of long Legendrian embeddings; i.e. where $\hat{\gamma}^k(t)=(t,0,0)$ for $|t|>1$ and all $k\in\NS^n$. We can thus consider the associated $n$-sphere of tangles $\{\tilde{\gamma}^k\}_{k\in\NS^N}$ defined as in Subsection \ref{LongKnotsTangles}; i.e. $\tilde{\gamma}^k:=\hat{\gamma}^k|_{[-1.1]}:[-1,1]\rightarrow (\D^3,\xi_\std)$. Clearly, the Legendrian closure (Definition \ref{Closure}) of this $n$-sphere of tangles yields an actual $n$-sphere of Legendrian embeddings (in other words, it is an $n$-sphere of embeddings, each of which is in standard position). Henceforth, we will say that such an $n$-sphere of Legendrian embeddings is an \textbf{$n$-sphere of embeddings in standard position}.

\begin{remark}
    When thinking of such an $n$-sphere in the Lagrangian projection, it may be useful to have in mind the following picture. We can represent it as an $n$-sphere taking place within a (small) box (as in Figure \ref{LongToShort}), in which interior it coincides with the associated $n$-sphere of tangles, and outside of which it coincides with the (constant) standard Legendrian unknot $\beta$. 
\end{remark}

Let us show how this particular type of spheres fit in the parametric connected-sum construction. Take an $n$-sphere $\gamma_1(t,k)$ of Legendrian embeddings in standard position together with another $n$-sphere of Legendrian embeddings $\gamma_2(t,k)$. Think the embedding $\gamma_1\#\gamma_2$ as the closure of the concatenation of the two associated tangles (Proposition \ref{StandardModelPosition}), which  can be respectively thought as lying within two (small) boxes. Let us state the following key remark regarding the diagrammatic representation of $\mathcal{E}_1(t,k)$ in this particular case.

\begin{remark}
If $\gamma_1(t,k)$ is an $n$-sphere of Legendrian embeddings in standard position, then, clearly by construction, the $n$-sphere $\mathcal{E}_1(t,k)$ can be represented, up to reescaling accordingly as in Remark \ref{rmk:Reparametrizations}, as the original sphere $\gamma_1(t,k)$ of embeddings in standard position (taking place within a small box) concatenated with the (constant) $\gamma_2$-tangle lying within its own separated box not interacting with it. 
\end{remark}

In other words, we can think the $n$-sphere $\mathcal{E}_1(t,k)$ as the $n$-sphere of $\gamma_1$-tangles taking place within its box and, outside of that box,  the $n$-sphere just coincides with the (constant) $\gamma_2$-tangle and the respective closure. 
Furthermore, note that the same box containing the $\gamma_1$-tangle just moves rigidly within the sphere $\mathcal{E}_2(t,k)$ in the second part of the parametric-connected sum construction. In other words, the $\gamma_1$-tangle lies within its own box throughout the whole construction of $\gamma_1\#\gamma_2(t,k)$. Therefore, it is clear that it is not relevant in which order we perform both spheres $\mathcal{E}_1(t,k)$ and $\mathcal{E}_2(t,k)$; i.e. they clearly commute since they do not interact with each other and can be performed simultaneously. We encapsulate this fact in the following Remark.

\begin{remark}\label{RemarkSP}
    
By the considerations above, when one of the $n$-spheres involved in the parametric connected sum is of Legendrian embeddings in standard position, it readily follows that the $n$-spheres 
$\mathcal{E}_i(t,k)$ ($i=1,2$) can take place at any preferred order (or even simultaneuosly) in its diagrammatic representation. Indeed, since they do not interact with each other, all the possible orders yield isotopic $n$-spheres (this is the case since obvious isotopies can be defined). 
\end{remark}

\begin{remark}\label{ExampleLoopIntro}
    See, for an explicit example of the above kind, the parametric connected sum loop in Figure \ref{loopsMix}. The connected sum of two right-handed Legendrian trefoils (in blue and green) plays the role of the $\gamma_1$-tangle, whereas the black big trefoil corresponds to the $\gamma_2$-component.
    
    In this case, $\mathcal{E}_1(t,k)$ can be regarded as taking place in a small box (which we have not depicted explicitly) where these coloured small trefoils lie. Note that, since $\gamma_1(t,k)$ is a loop of Legendrians in standard position, we have depicted the parametric connected-sum as both spheres $\mathcal{E}_1(t,k)$ and $\mathcal{E}_2(t,k)$ taking place simultaneuosly (following Remark \ref{RemarkSP}). More precisely, $\gamma_1(t,k)$ is a ``Pulling one knot through the other loop'', which will be introduced in Subsection \ref{familia2}.
\end{remark}

\section{LCH Monodromy invariant}

In this Section we introduce the Chekanov-Eliashberg's Differential Graded Algebra (DGA) of a Legendrian \cite{Chekanov,EliashbergContactInvariants} and its associated homology, named Legendrian Contact Homology (LCH). Finally, we review the LCH monodromy invariant of a $1$-parameter family of Legendrians introduced by \kalman in \cite{kalman}. We will follow the exposition in \cite{Chekanov,kalman}. The reader may also consult \cite{EtnyreNg}.

\subsection{Chekanov--Eliashberg's DGA and LCH} 

Let $K\in\widehat{\Leg}$ be an embedded Legendrian knot with generic Lagrangian projection $K_L$. We define its associated DGA $\mathcal{A}_K$ over $\Z_2$, which will be enough for our purposes.  The generators of the algebra are Reeb chords with end points in $K$, that is, the set of crossings $C=\lbrace c_1,\cdots, c_k\rbrace$ of its Lagrangian projection $K_L$. Let us define the grading. Following \cite{kalman}, we say that a capping path of a crossing $c\in K_L$  is a path in $K_L$ starting at the undercrossing and whose endpoint is the uppercrossing of $c$. Moreover, asumme that all the crossings take place at right angles and, therefore, the rotation number of the capping path $r_c$ takes values $(2k+1)/4, k\in\Z$. Define the grading of $c$ as 
\[|c|=-2r_c-\frac{1}{2}\in\Z\]
and extend it in the usual way to the whole algebra $\mathcal{A}_K$.
\begin{remark}\label{rotationRemark}
We will work all along this article with Legendrians $K$ with zero rotation number $\Rot(K)=0$. Note that, in this case, the grading is independent of the choice of the chosen orientation for the capping path and it is thus well defined. 
\end{remark}
Define the Reeb sign associated to a quadrant at a crossing by the following convention: following the counter clockwise orientation of the plane, at each quadrant we have two branches (forming the corresponding right angle)  ordered by such orientation. If the first one is the upper branch and the second one the under branch, then we assign a positive $+$ sign to that quadrant. If they come in reverse order at that quadrant, we assign it a negative $-$ sign (see Figure \ref{Fig:SignConvention}).

\begin{figure}[h]
	\centering
	\includegraphics[width=0.19\textwidth]{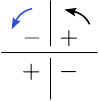}
	\caption{Sign convention for the quadrants at a crossing. Note that the quadrants where the first branch (following the order given by the counter clockwise orientation) is the upper branch and the second one is the under branch (black curved arrow) are equipped with a positive $+$ sign. The 
 other two quadrants where the branches come in opposite order (see the blue arrow) are equipped with a negative $-$ sign.\label{Fig:SignConvention}}
\end{figure}

%fix the ordered basis given by the tangent vector of the Legendrian at the upper crossing and the tangent vector at the under crossing. We say that the Reeb sign is positive if the orientation of this basis coincides with the standard one in $\R^2$ and it is negative otherwise. 

Fix crossings $p, c_1,\cdots, c_n$; denote by $\mathcal{P}(p;c_1,\cdots, c_n)$ the set of (non parametrized) $(n+1)$--polygons $P$ in $\R^2$ satisfying the following properties:

\begin{itemize}
\item[(i)] $P$ is immersed everywhere except at the ordered set of points $\mathcal{Q}=\{p, c_1,\cdots, c_n\}$.
\item[(ii)]  $\partial P\subset K_L$.
\item[(iii)]  The non immersed curve $\partial P$ fails to be immersed at the sequential set of points $\mathcal{Q}$. Moreover, it follows a positive (resp. negative) quadrant for $p$ (resp. for $c_j$).
\end{itemize}

Do note that $c_j$ and $c_{j'}$, for $j\neq j'$, may be the same crossing. The set of points $\{c_1,c_2,\ldots,c_n\}$ could be possibly empty.

Define the differential of a crossing $p$ via the formula:
\begin{equation}\label{differential}
\partial(p)=\sum_{n\in \N, (c_1,\cdots,c_n)\in C^n}\#\mathcal{P}(p;c_1,\cdots,c_n)c_1\cdots c_n
\end{equation}

and extend it to the rest of the words by the (graded) Leibniz rule. By convention the empty word is set to be $1$. Chekanov proved that the differential decreases the degree by $1$ \cite{Chekanov}.

\begin{remark}
    Note that, since we are working with $\mathbb{Z}_2$ coefficientes, we could potentially have different expressions for equivalent words. For example, $a+b+b=a$. We will, therefore, say that a word is \textbf{minimal} if it minimizes its length among all equivalent expressions. We will implicitly work with minimal expressions unless stated otherwise.
\end{remark}

\begin{figure}[h]
	\centering
	\includegraphics[width=0.3\textwidth]{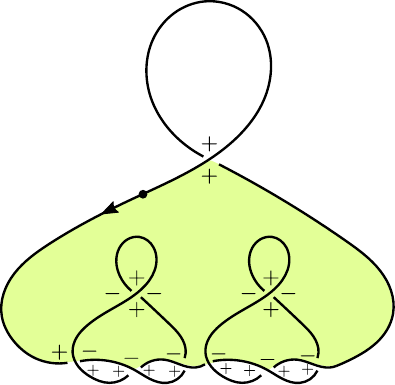}
	\caption{Example of a Legendrian connected sum of two trefoil knots with the described choice of signs for the crossings. In yellow, an example of one of the polygons contributing to the differential of the uppermost crossing in the diagram.\label{crossingsexample}}
\end{figure}

\begin{remark}\label{rmk:Heights}
	Define the height, or action, $h(c)>0$ to be the length of the Reeb chord $c$.
    
    If $P \in \mathcal{P}(p;c_1,\cdots,c_n)$, fix a parametrization of $P$ namely $\tilde{P}:\D^2\rightarrow\R^2$, then by Stokes' Theorem,
	\[
	h(p)-\sum_{i=1}^n h(c_i)= \int_{\D^2} \tilde{P}^*(dx\wedge dy) > 0.
	\]
	In particular, the differential $\partial$ decreases the height and, therefore, is well-defined. This last statement is encoded as a lemma since we will invoke it several times later on.
\end{remark}

\begin{lemma}\label{LemaAction}
If a crossing $x$ appears in the (minimal) expression of $\partial(y)$, then its action is smaller than the one of $y$; i.e. $h(x)<h(y)$. In particular, $y$ cannot appear in the expression of $\partial(x)$.    
\end{lemma}

\begin{theorem}[Chekanov \cite{Chekanov}]\label{thm:LegendrianContactHomologyWellDefined}
The pair $(\mathcal{A}_K,\partial)$ is a DGA. The homology $H_*(K)$ of the DGA, known as Legendrian contact homology (LCH) of $K$, is an invariant of the Legendrian isotopy class of $K$. 
\end{theorem}

Theorem \ref{thm:LegendrianContactHomologyWellDefined} was proved by Chekanov in a combinatorial way: he showed that the the DGAs $(\mathcal{A}_K,\partial)$ and $(\mathcal{A}_{K'},\partial')$ of two Legendrians $K$ and $K'$ that differ by a Reidemeister move in the Lagrangian projection move are stable tame isomorphic and, therefore, they have isomorphic LCH. To introduce and compute \kalmans monodromy invariant, we will make use of such isomorphisms.

%\begin{remark}\label{rmk:LegendrianContactHomologyWellDefined}
%The differential $\partial$ is well defined since this sum has a finite number of non--vanishing elements, because of the previous property. Moreover, it has degree $-1$ and its square is zero. The algebra $\mathcal{A}_K$ equipped with the boundary operator $\partial$ is known as Chekanov--Eliashberg DGA of the Legendrian  knot $K$. This DGA gives rise to the Chekanov-Eliashberg Legendrian Contact homology $H_*(K)=\ker (\delta)/\im(\delta)$ of a Legendrian $K$, that we will show to depend just on the (Legendrian) isotopy class of the Legendrian \cite{Chekanov}.
%\end{remark}

\subsection{Invariants for one-parameter families of Legendrians.}

Let $K^t$, $t\in[0,1]$, be a generic path of Legendrians. Four essentially different Legendrian Reidemeister moves ($\mathcal{R}-\text{II}$, $\mathcal{R}-\text{II}^{-1}$, $\mathcal{R}-\text{III}_a$, $\mathcal{R}-\text{III}_b$) may take place in the Lagrangian projection at finitely many times $\lbrace t_0,\cdots, t_k\rbrace$ (see Figure \ref{reidemeisters}), and no other bifurcation takes place for different values of the parameter $t$. 

Denote by $\left(\mathcal{A}_K, \partial\right)$ the differential algebra associated to a Legendrian $K=K^{t_i-\varepsilon}$ before one of the Reidemeister moves takes place and by $\left(\mathcal{A}_{K'}, \partial'\right)$ the one associated to a Legendrian $K'=K^{t_i+\varepsilon}$ just after such a Reidemeister move. Then, the Lagrangian projections $K_L$ and $K'_L$ only differ where the Reidemeister move has taken place. Therefore, there is a relabeling $p\mapsto p'\in\mathcal{A}_{K'}$ for all generators $p\in\mathcal{A}_K$ of the algebra not affected by the move. For the ones affected by the Reidemeister move we have to name the old and the new generators.

The notation is as follows. In the $\mathcal{R}-\text{II}$ moves the pair of generators that appear have opposite signs in the bounded region (see the first box in Figure \ref{reidemeisters}), the positive one is labeled as $x$ and the other one as $y$. For the $\mathcal{R}-\text{III}$ moves denote by $z$ the generator that remains fixed under the movement and by $x$ and $y$ the two generators of the moving branch. Moreover, $x$ is chosen to be one that has negative sign within the bounded region and $y$ the onw with positive sign. See Figure \ref{reidemeisters}.

As shown by Chekanov \cite{Chekanov}, associated to the bifurcation there is a holonomy homomorphism $$ g:(\mathcal{A}_K,\partial)\rightarrow (\mathcal{A}_{K'},\partial'),$$ which induces an isomorphism on the homology level. The description of the holonomy is the content of the following Theorem.

\begin{figure}[h]
	\centering
	\includegraphics[width=1\textwidth]{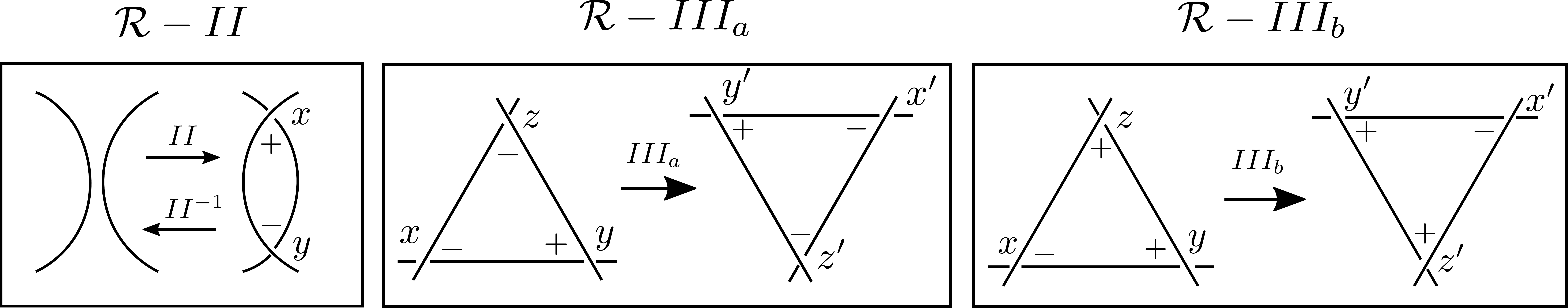}

\caption{Different Reidemeister moves that could potentially take place in the Lagrangian projection of a generic path of Legendrian embeddings.}\label{reidemeisters}
\end{figure}

\begin{theorem}[\cite{Chekanov}, \cite{kalman}]\label{theoremChekanov} %Let $K_t$, $t\in[0,1]$ be a generic path of Legendrians. Denote by $\lbrace t_0,\cdots, t_k\rbrace$ the times where some Legendrian Reidemeister move takes place in the Lagrangian projection of the path. The following algebra homomorphisms, called holonomy maps, $g:\mathcal{A}_{K_{t_i-\varepsilon}}\rightarrow\mathcal{A}_{K_{t_i+\varepsilon}}$ (described below for generators depending on the different Reidemeister move taking place at $t_i$)  induce isomorphisms at the homology-level.
Name the generators of $\mathcal{A}_K$ as follows ordered by their action/height:
\begin{equation}\label{heightsOrder} h(x_q)\geq h(x_{q-1})\geq\cdots\geq h(x) > h(y) \geq h(y_1) \geq h(y_2) \geq \cdots \geq h(y_\ell)\end{equation}
and write $\partial(x)=y+w$. The holonomy homomorphism $g$ is described, depending on the type of bifurcation, as follows:

\scalebox{0.7}{$\blacksquare$}  $\mathcal{R}-\text{II}$. If the two crossings appearing are $x$ and $y$ (and thus we have the crossings ordered by height as in (\ref{heightsOrder})), then the holonomy map  satisfies that $g(y_j)=y'_j$ and is defined recursively for the other generators $x_1,\cdots, x_q$ by a formula described in \cite[pp. 2038--2039]{kalman} that we will not use in this article. In practice, we will, instead, make use of the more tractable description provided by Proposition \ref{holonomiaTipoDos} below in order to deal with  $\mathcal{R}-\text{II}$ moves.

\scalebox{0.7}{$\blacksquare$}   $\mathcal{R}-\text{II}^{-1}$: $g(x)=0, g(y)=w'$; and  $g(p)=p'$ for any other $p\in\mathcal{A}_K$.

\scalebox{0.7}{$\blacksquare$}  $\mathcal{R}-\text{III}_a$: $g(p)=p'$ for any $p\in\mathcal{A}_K$.

\scalebox{0.7}{$\blacksquare$} $\mathcal{R}-\text{III}_b$: $g(x)=x'+z'y'$ and $g(p)=p'$ for any other generator $p\in\mathcal{A}_K$.

\end{theorem}

%Rather than working with the involved formula described for the holonomy of the $\mathcal{R}-\text{II}$ move, we will in practice use the following proposition.

\begin{proposition}[\cite{kalman}, Prop. 3.5.]\label{holonomiaTipoDos} 

Let $p\in\mathcal{A}_K$ be a generator such that, after the holonomy of a Reidemeister$-II$ move, it becomes $p'\in\mathcal{A}_{K'}$ and, moreover, $\partial(p')=0$. Then the holonomy $g:\mathcal{A}_K\rightarrow \mathcal{A}_{K'}$ acts as the identity on $p$; i.e. $g(p)=p'$.

%Let  $g:\mathcal{A}_K\rightarrow \mathcal{A}_{K'}$ be the holonomy of a Reidemeister$-II$ move and   after the move, its corresponding $p'\in\mathcal{A}_{K'}$  satisfies $\partial(p')=0$. Then the holonomy $g:\mathcal{A}_K\rightarrow \mathcal{A}_{K'}$ corresponding to a $\mathcal{R}-II$ move acts as the identity on $p$; i.e. $g(p)=p'$.
\end{proposition}

Given a generic loop $K^\theta\in\widehat{\Leg}$, $\theta\in\NS^1$, based at $K=K^N$, there is a finite number of bifurcations that appear. The composition of all the holonomy homomorphisms give rise to a stable tame isomorphism $g(K^\theta):(\mathcal{A}_K,\partial)\rightarrow(\mathcal{A}_K,\partial)$ which defines an element $\mu(K^\theta)\in \operatorname{Aut}(H_*(K))$,
known as the \textbf{ LCH monodromy} of the loop. This is an invariant of the homotopy class  of the loop, as the following Theorem states.

\begin{theorem}[\kalman \cite{kalman}]\label{MonodromyInvariant}
Let $K\in\widehat{\Leg}$ be a generic Legendrian. The monodromy map is a group homomorphism
\[\mu:\pi_1(\widehat{\Leg},K)\to \operatorname{Aut}(H_*(K)).\] 

\end{theorem}

\section{Parametric connected-sums and the LCH Monodromy invariant}\label{argumentomonodromia}
In the first part of the Section we will take a close  look to DGA of positive $(p,q)$-torus knots with maximal Thurston-Bennequin invariant. We will also introduce a new class of Legendrians, called knots of even $\partial$-class, that will play a crucial role in our construction. 

Then, we will prove a key result (Lemma \ref{BarridosNoAfectan}) about the monodromy of a ``fly'' in a connected sum loop. After reviewing \kalmans examples \cite{kalman}, we will apply the Elephant-Fly construction to prove Theorem \ref{thm:KalmanSumNonTrivialA} which will be the key to produce new infinite families of loops with non-trivial monodromy invariant in Section \ref{SecExamples}, obtained by considering parametric connected sums of loops based at Legendrians of even $\partial$-class with \kalmans loop of right-handed trefoils. 

When referring to loops of Legendrian embeddings, we will use $t\in\NS^1$ to denote the parameter in the domain of the embedding, and $\theta\in \NS^1$ for the parameter of the loop.

\subsection{Legendrian $(p,q)$-tangles and associated words.}\label{pqRealization}

Positive Legendrian $(p,q)$-torus knots were classified by Etnyre and Honda in \cite{EtnyreHondaTorus}. There is a unique Legendrian $(p,q)$-torus knot with maximal Thurston-Bennequin invariant $\tb=pq-p-q$ and rotation number $\rot=0$. We will refer to the latter as the Legendrian $(p,q)$-torus knot without specifying the condition about the $\tb$ number. We will follow \kalmans convention \cite{kalman} to refer to $(p,q)$-torus knots meaning that the knot wraps $p$ times around the meridian and $q$ times around the Seifert longitude of an unknot. 

A diagrammatic realization of the Legendrian $(p,q)-$torus knot in the Lagrangian projection consists of the following construction. Take the usual $(p,q)-$braid and close it by joining the $i$-th strand at the left side of the braid with the $i$-th strand at the right side of the braid via a closed unknotted curve with a kink (for every $1\leq i\leq q$) as in Figure \ref{torusknot} (A). Note that the total (signed) area enclosed by the Lagrangian projection must be $0$ but for practical reasons we will not take this into consideration in our Figures. We will denote this Legendrian $(p,q)$-torus knot as $K_{p,q}$. We denote by $\gamma_{p,q}$  a parametrization of $K_{p,q}$ that we assume to be in standard position.

The associated Legendrian $\gamma_{p,q}$-tangle is obtained by opening the Legendrian $\gamma_{p,q}$ at a point $Q$ in the curve joining the points in the $q-$th position both at left and right side of the Legendrian $(p,q)-$braid in the knot. We will refer to the $\gamma_{p,q}$-tangle as Legendrian \textbf{$(p,q)$-tangle}. See Figure \ref{torusknot} for an example of a $(4,3)$-torus knot (A) and its corresponding $(4,3)$-tangle (B). Sometimes, we will make a slight abuse of notation and talk about embedded tangles and embedded Legendrians in standard position (referring to the lifts of the so-called diagrammatic representations). 

%exactly as in \cite[Def 6.1]{kalman} or as in Figure \ref{torusknot}. Kálmán calls this construction the \textbf{Legendrian closure} of the braid \cite[p. 2053]{kalman}. See Figure \ref{torusknot} (A). The enclosed $0-$area condition must be satisfied when performing this diagrammatic realization but for practical reasons we will not take this into consideration in our Figures.

%\begin{remark}
%By the work of J. Etnyre and K. Honda \cite{EtnyreHonda}, Legendrian positive torus knots are classified by their formal invariants and so every (non-stabilized) Legendrian embedding of such type is isotopic to the model we have just described. 
%\end{remark}

%\begin{remark}
%The diagrammatic representation we just described can be understood as a Legendrian embedding in standard position (recall Definition \ref{StandardPosition}), where the $\gamma$-tangle in this case corresponds to the $(p,q)$-braid. We often call it the $(p,q)$-tangle.
%\end{remark}

%\begin{definition}
%We define a Legendrian $(p,q)-$tangle as a tangle consisting on opening a Legendrian $(p,q)-$torus knot at a point $Q$ in the curve joining the points in the $q-$th position both at left and right side of the Legendrian $(p,q)-$braid in the knot.
%\end{definition}
%See Figure \ref{torusknot} for an example of a $(4,3)$-torus knot (A) and its corresponding $(4,3)$-tangle (B).
%\begin{remark}
%Any election of point $Q$ gives raise, up to Legendrian homotopy, to the same tangle diagram and, therefore, we can speak about the $(p,q)-$tangle with no ambiguity.  
%\end{remark}
\begin{figure}[h]
		\centering
		\subfloat[$(4,3)$-torus knot and a choice of  point $Q$ in the curve.% joining the points in the $3-$th position both at left and right side of the Legendrian $(4,3)-$braid in the knot
  ]{
			\label{torusknot}
			\includegraphics[width=0.45\textwidth]{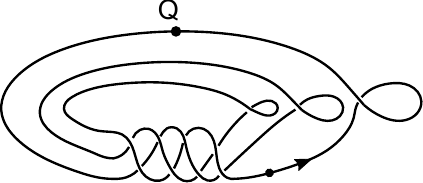}}
		\subfloat[Legendrian $(4,3)$-torus tangle.]{
			\label{torustangle}
			\includegraphics[width=0.45\textwidth]{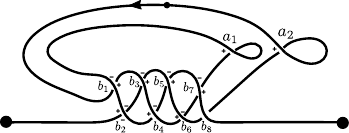}}
		\caption{$(4,3)$-torus knot and its associated tangle. }\label{torusknot}
\end{figure}

The Lagrangian projection of a Legendrian $(p,q)$-torus knot as above has degree $0$ crossings $b_1,\ldots, b_{p(q-1)}$ coming from crossing of the $(p,q)$-braid; and degree $1$ crossing $a_1,\ldots,a_q$ corresponding to the kinks. Following the notation from \cite{kalman}, degree $0$ crossings are ordered from left to right in the Lagrangian projection and degree $1$ crossings are ordered starting from the innermost crossing in the diagram until the outermost one.%.where $C$ denotes the set of crossings in the Legendrian $(p,q)$--braid inside the tangle. All these crossings have degree $0$ in the DGA \cite[p. 2053]{kalman}. Moreover, all the crossings corresponding to the $q$ upper loops have degree $1$ \cite[p. 2053]{kalman}. 

%\begin{remark} Following the notation from \cite{kalman}, we denote by $a_i$ the degree-$1$ crossings and by $b_j$ the degree-$0$ ones. Degree-$1$ crossings are ordered starting from the innermost crossing in the diagram until the outermost one. Degree-$0$ crossings are ordered from left to right in the Lagrangian projection (see the right picture in Figure \ref{torusknot}).
%\end{remark}

Note that the closure of a Legendrian tangle $K$ adds a new degree $1$ crossing $a$ to the Lagrangian projection (Figure \ref{LongToShort}). We refer to such $a$ as the \textbf{closure crossing}.

\begin{definition}\label{AssociatedWord}
 The word associated to the tangle $K$ is the differential $W_K:=\partial(a)-1$ %(recall Eq. \ref{differential}) of the closure crossing $a$ associated to the $K$-tangle, which results from taking the aforementioned canonical Legendrian closure: $W_K:=\partial(a)$.
\end{definition}

\begin{remark}
    This word will naturally appear when we consider additional tangles attached to given Legendrian embeddings. The differentials of certain crossings in the resulting new embeddings will be related to the original ones by means of these associated words. Let us justify the appearance of the $-1$ in this definition. It corresponds to the lobe that appears when one considers the closure of a tangle.
\end{remark}

%\begin{figure}[h]
%	\centering
%	\includegraphics[width=0.4\textwidth]{ClosureNewCrossing.pdf}
%	\caption{Any Legendrian tangle can be closed in a canonical way by taking its Legendrian closure as depicted in the picture. A new crossing $a$ appears in the diagram when this closure process is performed.\label{ClosureNewCrossing}}
%\end{figure}

\begin{figure}[h]
	\centering
	\includegraphics[width=0.5\textwidth]{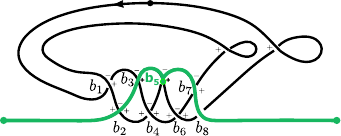}
	\caption{$b_5$ is one of the monomials in the word  $W_{K_{4,3}}$ associated to a $K_{4,3}$-tangle; i.e. $W_{K_{4,3}}=b_5+\cdots $.\label{torustangleword}}
\end{figure}

Let us elaborate a bit on the DGA of a $(n,2)$-torus knot $K_{n,2}$. 

\begin{definition}[\kalman \cite{kalman2}]
The \textbf{path matrix} associated to $K_{n,2}$ is defined as:

\[B^{n,2}:=\begin{pmatrix}
B_{1,1}^{n,2} & B_{1,2}^{n,2} \\
B_{2,1}^{n,2} & B_{2,2}^{n,2} \end{pmatrix}= \begin{pmatrix}
b_1 & 1 \\
1 & 0 
\end{pmatrix}\begin{pmatrix}
b_2 & 1 \\
1 & 0 
\end{pmatrix}\cdots\begin{pmatrix}
b_n & 1 \\
1 & 0 

\end{pmatrix},\]

where the entries $B_{i,j}^{n,2}$ just denote the polynomials defined by the product on the right.
\end{definition} 

\begin{remark}
In fact, Kálmán defines such a matrix for any positive braid but since we will focus on $(n,2)$-torus knots we will not introduce it in such generality.
\end{remark}

The path matrix is specially useful because of the following Lemma.

\begin{lemma}[Theorem 6.7 in \cite{kalman}]\label{diferencialesKalman}
For the Legendrian torus knot $K_{n,2}$ with degree $1$ generators $a_1, a_2$ and degree $0$ generators $b_1, \ldots, b_n$, we have:

\[\partial(b_i)=0,\quad \partial(a_1)=1+B_{1,1}^{n,2}\quad \text{ and } \quad \partial(a_2)=1+B_{2,2}^{n,2}+B_{2,1}^{n,2}B_{1,2}^{n,2},\]

where $B_{i,j}^{n,2}$ are the entries of the path matrix associated to $K_{n,2}$. %See also Example 3.3 in \cite{kalman2} for this computation.
\end{lemma}

\begin{example}[Example 3.3 in \cite{kalman2}]\label{example23}
For the positive right-handed Legendrian trefoil $K_{3,2}$ we have

\[ B^{3,2}= \begin{pmatrix}
b_1 & 1 \\
1 & 0 
\end{pmatrix}\begin{pmatrix}
b_2 & 1 \\
1 & 0 
\end{pmatrix}\begin{pmatrix}
b_3& 1 \\
1 & 0 
\end{pmatrix}=\begin{pmatrix}
b_1 + b_3 + b_1b_2b_3 & 1 + b_1b_2 \\
1 + b_2b_3 & b_2 \end{pmatrix}.\]

We obtain then that $\partial(a_1)=1+b_1 + b_3 + b_1b_2b_3$ and $\partial(a_2)=1+1+b_2+b_2b_3+b_1b_2+b_2b_3b_1b_2$.
\end{example}

\begin{definition}
The \textbf{length} of a polynomial $P$, denoted by $\ell(P)$, is the number of words in $P$.
\end{definition}

\begin{lemma}\label{fibonacci}
%Let $K_{n,2}$ be a torus knot with degree $1$ generators $a_1, a_2$ with a diagrammatic representation as described in the beginning of Subsection \ref{pqRealization}. Then, 
The lengths of the entries of the path matrix of a $(n,2)$-torus knot satisfy the following recursive property:
\[\ell(B_{1,1}^{n,2})=F_{n+1}, \quad \ell(B_{1,2}^{n,2})=\ell(B_{2,1}^{n,2})=F_n \text{ and } \quad \ell(B_{2,2}^{n,2})=F_{n-1},\]

where $F_i$ denotes the $i$-th \textbf{Fibonacci number}.
\end{lemma}
\begin{proof}
The proof follows by induction. On this induction we will not care if the value of $n$ gives raise to a torus knot or link. Note that the base case $n=3$ is proved in Example \ref{example23}. Therefore, we assume that the claim holds true for $n=k$ and we will prove it for $n=k+1$. Observe that
\[ B^{k+1,2}= B^{k,2} \cdot \begin{pmatrix}
b_{k+1} & 1 \\
1 & 0 
\end{pmatrix}=\begin{pmatrix}
B_{1,1}^{k,2} & B_{1,2}^{k,2} \\
B_{2,1}^{k,2} & B_{2,2}^{k,2} 
\end{pmatrix}\cdot \begin{pmatrix}
b_{k+1} & 1 \\
1 & 0 
\end{pmatrix}=\begin{pmatrix}
B_{1,1}^{k,2}\cdot b_{k+1}+B_{1,2}^{k,2}& \quad B_{1,1}^{k,2} \\
B_{2,1}^{k,2}\cdot b_{k+1}+B_{2,2}^{k,2} & \quad B_{2,1}^{k,2} 
\end{pmatrix}.\]
It follows that
\begin{itemize}

\item[(i)] $\ell(B_{1,1}^{k+1,2})=\ell(B_{1,1}^{k,2})+\ell(B_{1,2}^{k,2})=F_{k+1}+F_k=F_{k+2}$,
\item[(ii)] $\ell(B_{1,2}^{k+1,2})=\ell(B_{1,1}^{k,2})=F_{k+1}$,
\item[(iii)] $\ell(B_{2,1}^{k+1,2})=\ell(B_{2,1}^{k,2})+\ell(B_{2,2}^{k,2})=F_k+F_{k-1}=F_{k+1}$, 
\item[(iv)] $\ell(B_{2,2}^{k+1,2})=\ell(B_{2,1}^{k,2})=F_k$. 

\end{itemize}

This proves the inductive step, thus yielding the claim.\end{proof}

%\textbf{Base case: n=3.}  The computation is done in Example \ref{example23}.%The claim is true for $K_{3,2}$ (see Example \ref{example23}) since 
%
%\[\ell(B_{1,1})=F_{4}=3, \quad \ell(B_{1,2})=\ell(B_{2,1})=F_3=2, \quad \ell(B_{2,2})=F_{1}=1.\]

%\textbf{Inductive step: n=k+1.} We assume the claim true for $n=k$ and we will prove it for $n=k+1$. Note that $$

%Use the following notation:
%
%\[B^{2, k}_\beta=\begin{pmatrix}
%B_{1,1} & B_{1,2} \\
%B_{2,1} & B_{2,2} 
%\end{pmatrix},\quad
%B^{2, k+1}_\beta=\begin{pmatrix}
%\tilde{B}_{1,1} & \tilde{B}_{1,2} \\
%\tilde{B}_{2,1} & \tilde{B}_{2,2} 
%\end{pmatrix}.\]
%We then have

%\[ B^{2, k+1}_\beta= B^{2, k}_\beta \cdot \begin{pmatrix}
%b_{k+1} & 1 \\
%1 & 0 
%\end{pmatrix}=\begin{pmatrix}
%B_{1,1} & B_{1,2} \\
%B_{2,1} & B_{2,2} 
%\end{pmatrix}\cdot \begin{pmatrix}
%b_{k+1} & 1 \\
%1 & 0 
%\end{pmatrix}=\begin{pmatrix}
%B_{1,1}\cdot b_{k+1}+B_{1,2}& B_{1,1} \\
%B_{2,1}\cdot b_{k+1}+B_{2,2} & B_{2,1} 
%\end{pmatrix}.\]

%Therefore, it follows that

%\begin{itemize}

%\item[i)] $\ell(\tilde{B}_{1,1})=F_{k+1}+F_k=F_{k+2}$,
%\item[ii)] $\ell(\tilde{B}_{1,2})=F_{k+1}$,
%\item[iii)] $\ell(\tilde{B}_{2,1})=F_k+F_{k-1}=F_{k+1}$, 
%\item[iv)] $\ell(\tilde{B}_{2,2})=F_{k}$. 

%\end{itemize}

%This proves the inductive step, thus yielding the claim.\end{proof}

\subsection{Legendrians of $\partial$-even class}

We now introduce a class of Legendrians that will play a central role along the paper.

\begin{definition}\label{EvenDeltaClass}
A Legendrian knot $K$ is of \textbf{even $\partial$-class} if it is representable as a Legendrian in standard position where:

\begin{itemize}
    \item[i)] $rot(K)=0$,
    \item[ii)] $\ell(\partial(c))\in 2\Z$ for every degree $1$ generator $c\in\mathcal{A}_K$ and 
    \item[iii)] there are no negative degree crossings.
    \end{itemize}
\end{definition}
%, working in the Lagrangian projection, the differential of every degree-$1$ crossing $a_i$ is an expression $\partial(a_i)$ containing an even number of words.

\begin{remark}
    Let us say a few words about the conditions in Definition \ref{EvenDeltaClass}: $i)$ is just required for technical reasons all along the article (recall Remark \ref{rotationRemark}). Points $ii)$ and $iii)$ are properties that: a) are preserved under connected-summation (Proposition \ref{ConnectedSumsEvenDeltaClass}) and b) altogether imply that the LCH of the knot is not trivial under an additional mild assumption (Remark \ref{NonVanishingLCH}). Moreover, we will make use of property $ii)$ in our subsequent computations.
\end{remark}

Henceforth, when talking about knots of even $\partial$-class we will implicitly assume that they are in standard position. Likewise, when considering connected-sums of even $\partial$-class knots, we will assume that the connnected-sum is diagrammatically performed by concatenating their associated tangles as in Proposition \ref{StandardModelPosition}.

The following remark has interest on its own.

\begin{remark}\label{NonVanishingLCH}
    Note that if $K$ is of even $\partial$-class and, in addition, contains a $0$-degree generator $b\in \mathcal{A}_K$ (thus with $\partial(b)=0$), then the knot has non-trivial level-$0$ homology $H_0(K)\neq 0$. The reason is that all differentials $\partial(W)$ of degree-$1$ words $W$ live in the ideal generated by differentials of degree-$1$ crossings $I=\langle \partial(a_1),\cdots \partial(a_k)\rangle $. Therefore, since the monomial $b$ has odd-length $\ell(b)=1\notin 2\mathbb{Z}$, then $[b]\in H_0(K)$ represents a non-trivial homology class.
\end{remark}

The following proposition shows that there are infinitely many Legendrian torus knots of even $\partial$-class and thus the study of 
this class of knots is meaningful.
\begin{proposition}\label{EvenDeltaClassTorus}
Legendrian positive $(n, 2)$ torus knots are of  even $\partial$-class if and only if $n\not\equiv 2 \mod 3$.
\end{proposition}
\begin{proof}
As explained in the previous Subsection conditions $i)$ and $iii)$ in Definition \ref{EvenDeltaClass} are true for positive torus knots. So it suffices to check property $ii)$. By Lemma \ref{diferencialesKalman} and Lemma \ref{fibonacci} we know that $\ell\left(\partial(a_1)\right)$ has the same parity as the number $A_1:=F_{n+1}-1$ while $\ell\left(\partial(a_2)\right)$ has the same parity as the number $A_2:=F_{n-1}+F_n^2-1$. 

The $i$-th Fibonacci number $F_i$ is even if and only if $i\equiv 0 \mod 3$. Therefore, if $n\equiv 0\mod 3$ or $n\equiv 1 \mod 3$, both $A_1$ and $A_2$ are even numbers, while if $n\equiv 2\mod 3$ then this is not the case, thus yielding the claim.\end{proof}

We can construct infinitely many more examples by the following corollary.

\begin{proposition}\label{ConnectedSumsEvenDeltaClass}
If $K_1,\ldots, K_k$ are of even $\partial$-class, then their connected sum $K_{1}\#\cdots\# K_{k}$ also is.
\end{proposition}

\begin{proof} Properties $i)$ and $iii)$ in Definition \ref{EvenDeltaClass} are straightforward. As for $ii)$, each of the knots $K_j$ is individually of even $\partial$-class. Then, each associated word $W_{K_j}$ is of odd length (recall Definition \ref{AssociatedWord}); i.e. $\ell(W_{K_j})\in 2\mathbb{Z}+1$, $j=1,\cdots, k$. 

Recall that we consider the connected sum as the closure of the concatenation of their respective associated tangles. Observe that the closure crossing  $a$ (i.e. the uppermost crossing in the diagram) has differential 

\[\partial(a)=1+W_{K_{1}} W_{K_{2}}\cdot \cdots W_{K_{k}}\] 

and, thus, its length $\ell(\partial(a))$ is even. 

Additionally, consider any other degree $1$ crossing $x$ in any of the knots $K_{j}$ for some $j=1,\cdots, k$, prior to performing the connected sum. Since $K_{j}$ is of even $\partial$-class, then   $\partial(x)\in 2\mathbb{Z}$. Also, note that after performing the connected sum with the rest of the knots, the differential of this crossing $x$ remains the same, i.e. it coincides with its original expression before performing the connected sum.

In order to check this last statement, note that each of the tangles associated to $K_{j,2}$ can be made arbitrarily small via a shrinking isotopy. Therefore, by Lemma \ref{LemaAction}, the expression of the differential of any crossing in it cannot involve crossings from the other tangles since they have much higher action. This yields the claim.
\end{proof}

\begin{corollary}\label{ConnectedSumsN2}
    
Connected sums $K=K_{n_1, 2}\#\cdots \# K_{n_k, 2}$ of Legendrian positive torus knots  $K_{n_i, 2}$ are of  even $\partial$-class if $n_i\not\equiv 2 \mod 3$ for all $i=1,\ldots, k$.
\end{corollary}

\begin{definition}
Consider a certain knot diagram with set of crossings $C$. For any formal expression 

\[\phi=\sum_{k\in\mathbb{N}, (b_1,\ldots, b_k)\in C^k}\lambda(b_1,\ldots,b_k) b_1\cdots b_k,\]

where $\lambda(b_1,\ldots,b_j)\in\Z_2$, denote by $\max_{b_j}(\phi)\in\N$ the maximum number of $b_j$-letters appearing in a single word of the expression $\phi$. Define the function 

$$\tau_{b_j}(\phi)\in\N$$ 

to be the number of words in $\phi$ containing exactly $\max_{b_j}(\phi)$ $b_j$-letters.%(among all the words involved in $\phi$). Denote by $h_{b_j}^r(\phi)$ the number of words in $\phi$ containing exactly $r$ $b_j$-letters.We define the function $\tau_{b_j}(\phi)$ as follows:\[\tau_{b_j}(\phi)=h_{b_j}^{\max_{b_j}(\phi)}(\phi)\] 
\end{definition}

\begin{example}
Consider $\phi= b_1b_2b_1b_3b_1+b_1b_3b_1+b_1b_4b_1b_2b_1+b_4+b_2b_1$. Then, $\max_{b_1}(\phi)=3$, since the maximum number of $b_1$-letters appearing in a single word is $3$; attained by the first and the third words in the expression.
On the other hand, $\tau_{b_1}(\phi)=2$ since there are exactly two words in $\phi$ containing exactly $\max_{b_1}(\phi)=3$ $b_1$-letters.
\end{example}

In the following Lemma we label the crossings of the right-handed trefoil $K_{3,2}$ as in the previous subsection: $b_1,b_2$ and $b_3$ stand for the degree $0$ crossings; and $a_1$ and $a_2$ for the ones of degree $1$. The following observation will be crucial later.

\begin{lemma}\label{PropiedadesIdealDiferencial}
Let $\tilde{K}$ be a Legendrian knot of even $\partial$-class and consider the (also of even $\partial$-class) Legendrian $K:=\tilde{K}\# K_{3,2}$.
Any degree-$0$  element $\phi\in \im(\partial)$ satisfies that 

$$\tau_{b_3}(\phi)\in 2\Z.$$%Let $K$ denote the connected-sum knot $\tilde{K}\# K_{3,2}$ where $\tilde{K}=K_{p_1,q_1}\#\cdots\#K_{p_n,q_n}$ is the connected sum of $n$ torus knots which is of even $\partial$-class. Then for any element $\phi\in \im(\partial)\subset I=\langle \partial(a_1), \partial(a_2), \ldots, \partial(\tilde{a}_1),\ldots, \partial(\tilde{a}_k)\rangle$ (where $\tilde{a}_i$ are the degree-$1$ generators in the $\tilde{K}$-tangle and $a_1, a_2$ the other two degree-$1$ generators, see Figure \ref{trefoilconpq}), the following property holds:
%
%\begin{itemize}
%\item The integer $\tau_{b_3}(\phi)$ is even.
%\end{itemize}
%
\end{lemma}
\begin{proof}
We label the degree $1$ generators of $\mathcal{A}_K$ in the $\tilde{K}$-tangle as $\tilde{a}_1,\ldots,\tilde{a}_k$; and the other two generators as $a_1$ and $a_2$, where $a_2$ corresponds to the closure crossing (see Figure \ref{concatenationLoops}).
It is enough to check the property for every element in the ideal $I=\langle \partial(a_1),\partial(a_2),\partial(\tilde{a}_1),\ldots,\partial(\tilde{a}_k) \rangle$. If there was no $\tilde{K}$-tangle involved we would recover the trefoil knot, where we have (Example \ref{example23}):

\begin{align*}
\partial(a_1) =& 1+b_1+b_3+b_1b_2b_3\\
\partial(a_2) =& 1+1+b_2+b_2b_3+b_1b_2+b_2b_3b_1b_2. 
\end{align*}

\begin{remark} We write $\partial(a_2)=1+1+\cdots$  even if we are working in $\mathbb{Z}_2-$coefficients since when considering the additional $\tilde{K}$-tangle in the next case, one of the monogones will become another polygon and produce a different word.\end{remark}

\begin{figure}[h]
	\centering
	\includegraphics[width=0.4\textwidth]{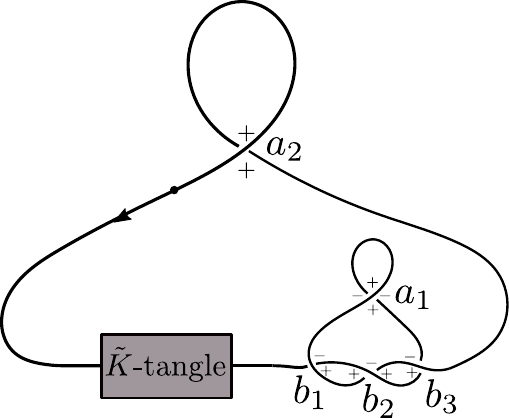}
	\caption{The Legendrian $K=\tilde{K}\#K_{3,2}$ with the labelling of some crossings.}\label{concatenationLoops} 
\end{figure}

%\begin{figure}[h]
%	\centering
%	\includegraphics[width=0.3\textwidth]{TrefoilConPQ.pdf}
%	\caption{Connected sum of a $(p,q)$-torus knot with a trefoil and corresponding labelling of some of the crossings.\label{trefoilconpq}}
%\end{figure}

Since we do have an additional $\tilde{K}$-tangle involved in the knot between the crossing $a_2$ and the corresponding path joining it with $b_1, b_2, b_3$, then the $\tilde{K}$-word $W_{\tilde{K}}$ appears in the expression as follows:
\begin{align*}
\partial(a_1) =& 1+b_1+b_3+b_1b_2b_3,\\
\partial(a_2) =& 1+W_{\tilde{K}}\cdot\left(1+b_2+b_2b_3+b_1b_2+b_2b_3b_1b_2\right).
\end{align*}

Also, it is clear that $b_1, b_2, b_3$ do not appear in the expressions  $\partial(\tilde{a}_i)$ for any of the $\tilde{a}_i$, by action considerations (recall the argument in the last paragraph of the proof of Proposition \ref{ConnectedSumsEvenDeltaClass}).
   
We have that 
\begin{align*}
\tau_{b_3}\left(\partial(a_1)\right)=& 2\in2\mathbb{Z},\\
\tau_{b_3}\left(\partial(a_2)\right)=& 2\cdot\ell\left(W_{\tilde{K}}\right)\in2\mathbb{Z}.
\end{align*}

Therefore the property holds for the generators of the ideal $I$. Since $\tilde{K}$ is of even $\partial$-class, all the expressions $\partial{\tilde{a}_i}$ contain an even number of words (Corollary \ref{ConnectedSumsN2}) and, thus, the property holds for any element in the ideal $I$.%both $\partial(a_i)$ generators and it also holds trivially for the other $\partial(\tilde{a}_i)$ generators. Since $W_{\tilde{K}}$ is of even $\partial$-class, all the expressions $\partial{\tilde{a}_i}$ contain an even number of words and, thus, it is clear then that the Property holds for any expression in the ideal $I$. Since $\im(\partial)\subset I$, the claim follows.
\end{proof}

\subsection{The monodromy of a fly}

The following lemma will be key in our subsequent arguments and has interest on its own. It encapsulates the idea that in the elephant-fly construction the monodromy of the fly is trivial; i.e. it is not affected by the movements of the elephant. See Remark \ref{EkholmRemark}. 

\begin{lemma}\label{BarridosNoAfectan}
Let $\gamma_1(t)$ be a Legendrian embedding realising a knot $K$ in standard position (with associated $K$-tangle) and let $\gamma_2^\theta(t)$ be any other loop of Legendrian embeddings. Consider the loop $\mathcal{E}^\theta_2(t)=\gamma\#\gamma_2^\theta(t)$, where the $K$-tangle is the fly. Then,

\begin{itemize}
    \item[i)] the holonomy $g:\mathcal{A}\to\mathcal{A}'$ of any Reidemeister move taking place for the loop $\mathcal{E}^\theta_2(t)$ acts as the identity for every crossing $p$ in the fly (the $K$-tangle); i.e. $g(p)=p'$.
    \item[ii)] the monodromy automorphism of the loop  $\mathcal{E}^\theta_2(t)$ acts on the crossings $p$ of the fly as the identity map; i.e. $g(p)=p'$.
\end{itemize}
\end{lemma}
\begin{proof}

For checking $i)$, we must check that when a branch of the loop passes over (or below) the $K$-tangle, the $0$-degree crossings of the $K$-tangle are mapped trivially by the holonomy morphism. In particular, we will show that this is the case for each of the elementary moves appearing in the homotopy.

\begin{figure}[h]
    \centering
    \subfloat{
        \label{fig:Step2}
        \includegraphics[width=0.4\textwidth]{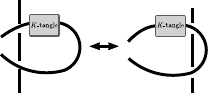}}
    \hspace{0.8cm} % Ajusta el espacio horizontal aquí
    \subfloat{
        \label{fig:Step3}
        \includegraphics[width=0.4\textwidth]{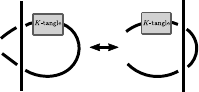}}
    \caption{Branch passing under and over a $K$-tangle, respectively.}\label{branchBlock}
\end{figure}
$\mathcal{R}-\text{III}_a$ moves: this follows automatically since it maps all generators trivially. 

$\mathcal{R}-\text{III}_b$ moves: note that since the moving branch in any $\mathcal{R}-\text{III}_b$ move corresponds to the elephant knot, the crossings of the $K-$tangle can only play the role of the crossings $y$ (if we think the branch $xz$ as the one that is moving) or $z$ (if we think the branch $xy$ as the one that is moving) in diagram (B) of Figure \ref{reidemeisters} and, thus, are mapped trivially as well. 

$\mathcal{R}-\text{II}$ moves:
 first, note that by applying a shrinking isotopy beforehand we can make the $K$-tangle sufficiently small and assume without loss of generality that the action of all the crossings in the $K$-tangle (fly) are arbitrarily small with respect to the two new crossings created. Therefore, by the description of $II$-type move holonomies from Theorem \ref{theoremChekanov} we conclude that the respective holonomies act as the identity on crossings in the $K$-tangle.

%this fact follows from an application of Proposition \ref{holonomiaTipoDos} together with the fact that the differential of every $0-$degree crossing in a $K$-tangle is always zero. 

$\mathcal{R}-\text{II}^{-1}$ moves: holonomies for $\mathcal{R}-\text{II}^{-1}$ are the identity for all points except for two points ($x$ and $y$ crossings in the first diagram of Figure \ref{reidemeisters}). But note that, in our case, these two points do not correspond to generators in the $K$-tangle when a branch passes over (or below) the block. Therefore all the cases are covered and the claim in $i)$ follows.

Note that $ii)$ is an obvious consequence of $i)$ by just noting that there is no non-trivial relabelling of the crossings in the fly at the end of the loop $\mathcal{E}^\theta_2(t)=\gamma\#\gamma_2^\theta(t)$ since the box moves rigidly (in the diagram) all along the loop.
\end{proof}

\begin{remark}\label{EkholmRemark}
	An intuitive/heuristic idea to justify why the claim in the previous Lemma should be true consists of observing that the action (Remark \ref{rmk:Heights}) of any crossing in the fly knot can be taken to be arbitrarily small compared to the action of any crossing in the elephant knot. This justifies why the Reeb chords corresponding to such crossings should not be able to ``reach'' the ones in the elephant knot.  We are thankful to Tobias Ekholm for this idea, which led us to believe that this statement might be true in a preliminary approach to the problem and, as a final outcome, took the form of Lemma \ref{BarridosNoAfectan}.

\end{remark}

\subsection{Kálmán's loops}

T. \kalman constructed a series of loops of Legendrian positive torus knots in \cite{kalman}, which he showed to have non-trivial monodromy invariant. He defined loops based at each $K_{p,q}$ positive torus knot. We will focus our attention on the trefoil $K_{3,2}$ case. So, let us provide a description of \kalmans loop for this case.

Consider first the loop $\LL_{3,2}(t,\theta)$ described in Figure \ref{fig:kalman}, which consists on taking a Legendrian trefoil $K_{3,2}$ knot and taking the $2$ strands of the knot to the 2-times cyclic rotation of them. Since we are considering Legendrian embeddings instead of just knots, we have considered \kalmans loop together with a loop of based points in Figure \ref{fig:kalman}. We have considered this particular parametrization for convenience, as this will be apparent later on. Henceforth, when we write $\LL_{3,2}(t,\theta)$ we will refer to this particular loop of Legendrian embeddings.

\begin{remark}[Notation]\label{loopsMpowers} Given a loop $\gamma(\theta,t)=\gamma^\theta(t)$ and $m\in\Z$ we will denote by $\gamma^m(\theta,t)=\gamma^{\theta,m}(t)$ its $m$-power inside the fundamental group of the space of Legendrian embeddings. %We will denote by $\LL^m_{3,2}(t,\theta)$ the $m$-times concatenation of the loop. Additionally, we write $\LL^{-1}_{3,2}(t,\theta)$ for the inverse loop and $\LL^{-m}_{3,2}(t,\theta)$ for the $m$-times concatenation of the inverse loop. Finally, $\LL^{0}_{3,2}(t,\theta)$ denotes the constant loop. Note that, this way, we have well defined loops $\LL^{i}_{3,2}(t,\theta)$ for any $i\in\mathbb{Z}$.
\end{remark}

We want to emphasize that $\LL_{3,2}(t,\theta)$ is just the loop $\Omega_{3,2}^2$ in \cite[p. 2015 and Sect. 9]{kalman} equipped with a parametrization for each knot in the loop (which varies smoothly with the parameter $\theta$). He showed that a whole family of loops (among which this one is a particular instance) are Legendrian non-trivial despite admitting a parametrization that makes them smoothly trivial. See \cite[4.2]{FMP} for a comprehensive discussion about the set of all possible parametrizations of these loops.

\begin{figure}[h]
	\centering
	\includegraphics[width=0.8\textwidth]{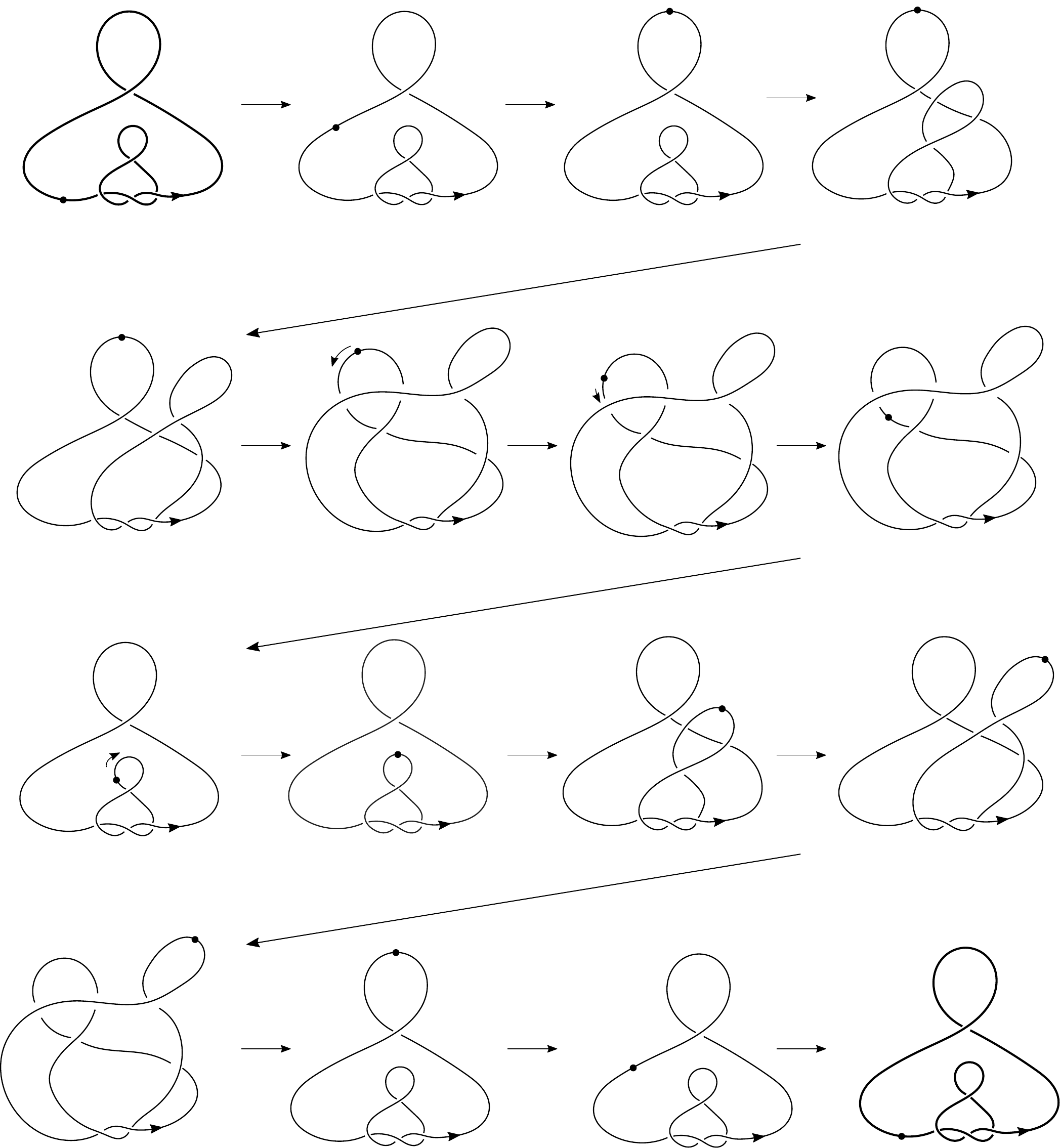}
	\caption{The loop $\LL_{3,2}(t,\theta)$ consisting on taking a Legendrian trefoil $K_{3,2}$ knot and taking the $2$ strands of the knot to the 2-times cyclic rotation of them. The base point of each embedding is depicted as a black dot.}\label{fig:kalman}
\end{figure}

\subsection{Loops with non-trivial monodromy}\label{LoopsExamples}

\begin{theorem}\label{thm:KalmanSumNonTrivialA}
Let $K$ be an even $\partial$-class Legendrian knot and $\mathcal{K}(t,\theta)$ a loop of Legendrians based on $K$. Then, for every $j\in\{1,2,3\}$, $$ \mu (\mathcal{K}\#\LL_{3,2}^j(t,\theta) )\neq \Id.$$%Then the monodromy invariant $\mu$ restricted to the $H_0$ homology-level, $\mu|_{H_0}$, is not the identity for the following cases:
%\begin{itemize}
%\item[(i)] $\mu|_{H_0}\left(\mathcal{K}\#\LL_{3,2}(t,\theta)\right)\neq \operatorname{Id}$.
%\item[(ii)]  $\mu|_{H_0}\left(\mathcal{K}\#\LL_{3,2}^2(t,\theta)\right)\neq \operatorname{Id}$.
%\item[(iii)]  $\mu|_{H_0}\left(\mathcal{K}\#\LL_{3,2}^3(t,\theta)\right)\neq \operatorname{Id}$.
%\end{itemize}
\end{theorem}

\begin{proof} 

\begin{figure}[h]
	\centering
	\includegraphics[width=1\textwidth]{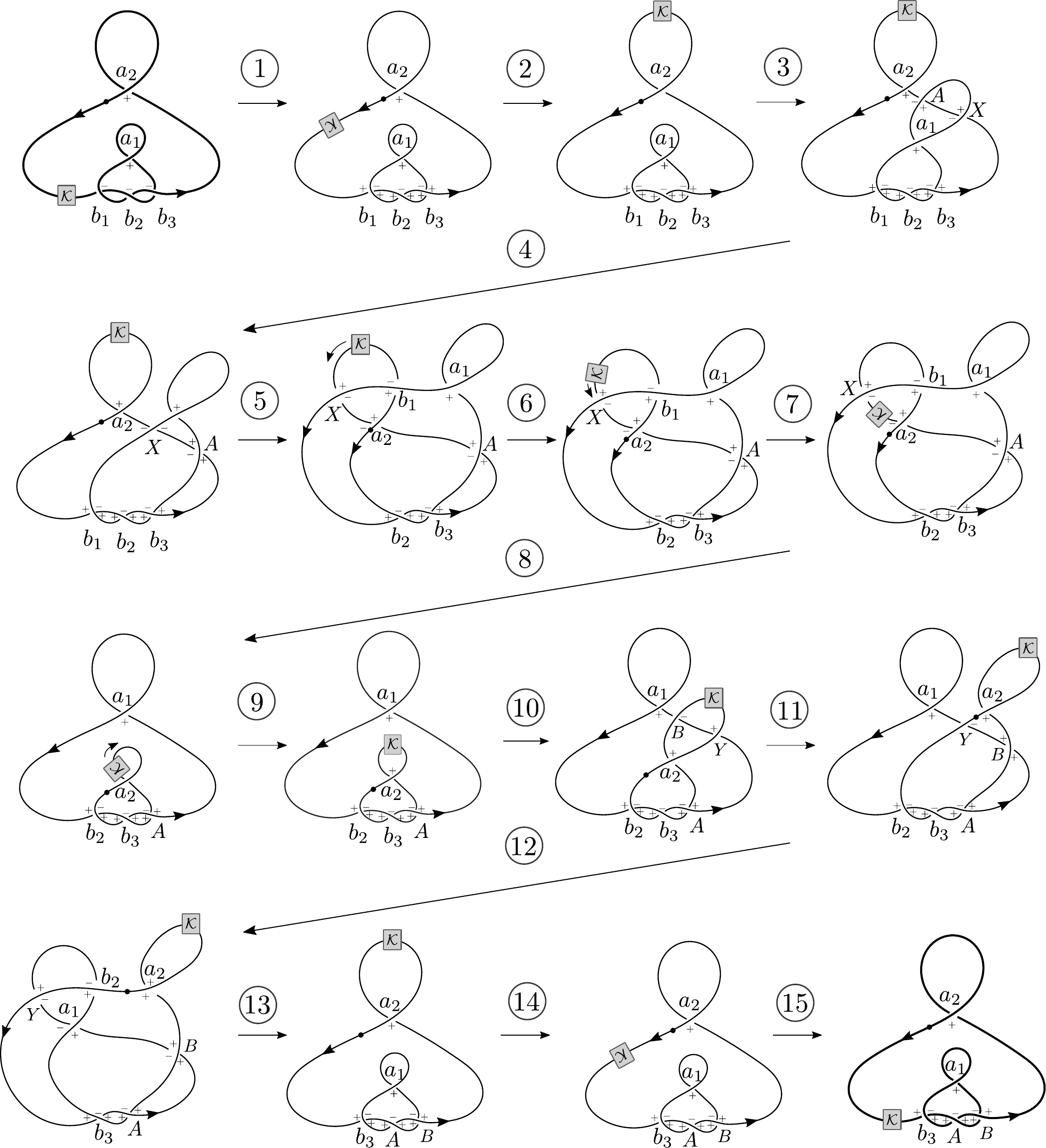}
	\caption{Half the loop of $\mathcal{K}\#\LL_{3,2}(t,\theta)$, where the $K_{3,2}$ trefoil plays the elephant role and the $K$-tangle (depicted as a tiny grey box) plays the role of the fly. Compare it with Figure \ref{fig:kalman}.\label{loopwithbox}}
\end{figure}
We will see that the monodromy restricted to $H_0(K\#K_{3,2})$ is not the Identity.
Precisely, we will focus on the monodromy map restricted to the degree-$0$ generators $b_1, b_2, b_3$ of the trefoil-tangle. A crucial observation is that by Lemma \ref{BarridosNoAfectan} these generators are trivially mapped to themselves after the first part of the loop (when the $(3,2)$-tangle plays the role of the fly). Therefore, it suffices to show that the monodromy acting on this generators is non-trivial only for the second part of the loop where $K$ plays the role of the fly (shown in Figure \ref{loopwithbox}). 

In order to check this, we will examine, one by one, the holonomies taking place at each arrow (we recommmend the reader to simultaneously look at Figure \ref{loopwithbox} while following the explanation below):

\circled{1}: There are no Reidemeister moves taking place.

\circled{2}: The second sequence of moves involves a bunch of II, $\text{II}^{-1}$, III$_a$ and III$_b$-moves but neither of them affect the $0$-degree generators in the $(3,2)-$tangle (i.e. they act as the identity on them). In order to check this, note that $\text{II}^{-1}$ and III-moves do not act on generators remotely. As for II-moves, it follows from Lemma \ref{holonomiaTipoDos}, since $\partial(b'_i)=0$ for $i=1,2,3$ after any of the moves involved. The fact that $\partial(b'_i)=0$ follows by visual inspection but, alternatively, one could argue as follows. Note that we could have applied a shrinking isotopy to the trefoil tangle at the beginning thus allowing for its crossings to have arbitrarily small action. Therefore, their differentials cannot involve crossings in the fly, concluding that $\partial(b'_i)=0$ for $i=1,2,3$.

\circled{3} and \circled{4}: II and $\text{III}_b$-moves in these two arrows do not affect these generators for the same reason. 

\circled{5}: This corresponds to a $\text{III}_b$-move that acts as the identity on $b_1$.

\circled{6}: There are no Reidemeister moves taking place.

\circled{7}: As in \circled{2} it involves several Reidemeister moves, all of which act as the identity on the $0$--degree crossings of the $(3,2)-$tangle. This is obvious for $II^{-1}$ and $III$-moves but we will check it carefully for $II$-moves as well.

It is clear that $\partial(b_1)=0$ before and after any of these moves and, so, Proposition \ref{holonomiaTipoDos} applies for generator $b_1$. On the other hand, note that both $b_2$ and $b_3$ appear in the expression $\partial(a_2)$ before and after any of these moves. Then by Lemma \ref{LemaAction} $a_2$ cannot appear in any of the expressions of the differentials $\partial(b_2), \partial(b_3)$. The only possibility, thus, is that $\partial(b_2)=0$ and $\partial(b_3)=0$ (both before and after any of these moves). By Proposition \ref{holonomiaTipoDos} we conclude that these holonomies act as the identity on $b_2$ and $b_3$ as well.

\circled{8}: It corresponds to a $\text{II}^{-1}$-move whose holonomies restricted to $0$-degree crossings coincide with the holonomies of this move in Kálmán's original loop (see \cite[Section 5]{kalman}) with an extra factor $W_{K}$ in $b_1$. In order to check this, just note that one of the crossing disappearing via this $\text{II}^{-1}$-move will map to zero whereas $b_1$ will be mapped to the differential of the other crossing before the move. Therefore, the $0$-degree crossings are mapped as follows:

\begin{align*}
b_1 &\mapsto \left(1+b_3A\right)W_{K}\\
b_2 &\mapsto b_2\\
b_3 &\mapsto b_3
\end{align*}

\circled{9}: There are no Reidemeister moves taking place.

\circled{10}: Several Reidemeister moves take place but they do not affect $0$-degree generators of the elephant for similar reasons as in arrow \circled{7}. Let us allaborate on this.

$III$ and $II^{-1}$-moves do not act ``remotely'' on crossings not involved in the moves and, thus, all these holonomies act as the identity on degree-$0$ generators of the elephant. We just have to check that this is the case for $II$-moves too.

Note that the expression $\partial(a_2)$ contains crossings $b_2$, $b_3$ and $A$ both before and after any of these moves. Then Lemma \ref{LemaAction} implies that neither of these $0$-degree crossings can contain $a_2$ in their differential. Therefore, the only possibility is that $\partial(b_2)=\partial(b_3)=\partial(A)=0$ (both before but, in particular, after the moves) and, by Proposition \ref{holonomiaTipoDos} we conclude that $II$-moves in this arrow act as the identity map as well for degree-$0$ generators of the elephant.

\circled{11}: It corresponds to a III$_b$-move not affecting $0$-degree crossings of the elephant.

\circled{12}: This corresponds to a III$_b$-move that maps $b_2$ to $b_2$ and also acts as the identity over any other $0-$degree crossings in the elephant.

\circled{13}: It corresponds to a $\text{II}^{-1}$-move that maps $b_2$ to $1+AB$ and acts as the identity for the rest of the $0-$degree crossings. This computation coincides with the one for the $\text{II}^{-1}$-move in \cite[Example 3.6]{kalman}.

\circled{14} and \circled{15}: The moves taking place here do not affect the $0$-degree crossings of the elephant since they all have $0$-differential before and after the moves and, thus, Proposition \ref{holonomiaTipoDos} applies.

\begin{remark}\label{HolonomiesTangleSmall} Recall that every crossing in the word $W_K$ is mapped trivially to itself while $K$ plays the role of the fly by Lemma \ref{BarridosNoAfectan} and, thus, $W_K$ is mapped to itself as well for this half of the loop.
\end{remark}

Note that at the end of this part of the loop $b_3$ takes the original position of $b_1$, $A$ the one of $b_2$ and $B$ the one of $b_3$.  Therefore, the monodromy of $\mathcal{K}\#\LL_{3,2}(t,\theta)$ restricted to the $0$-degree generators of the $(3,2)-$tangle is described as follows:

\begin{align*}
b_1 &\mapsto W_{K}+b_1b_2W_{K}\\ 
b_2 &\mapsto 1+b_2b_3\\ 
b_3 &\mapsto b_1.
\end{align*}

All this implies that the monodromy map acts on generator $b_3$ as $\mu(b_3)=b_1$. Since $\tau_{b_3}\left(\mu(b_3)-b_3\right)=1$, we conclude that $[\mu(b_3)-b_3]\in H_0(K\#K_{3,2})$ is not the zero by Lemma \ref{PropiedadesIdealDiferencial}. This proves the result for $j=1$.

For the loop $\mathcal{K}\#\LL_{3,2}^{2}(t,\theta)$ we iterate the previous relations to obtain that the monodromy acts on $b_3$ as 
\begin{align*}
b_3 &\mapsto W_{K}+b_1b_2W_{K}.
\end{align*}

Note that we are using again that all the holonomies act on $W_K$ as the identity as we iterate this part of the loop (Remark \ref{HolonomiesTangleSmall}).

Lemma \ref{PropiedadesIdealDiferencial} implies that $[\mu(b_3)-b_3]\in H_0(K_{p,q}\#K_{3,2})$ is not zero, which proves the result for $j=2$.

Finally, for the loop $\mathcal{K}\#\LL^3_{3,2}(t,\theta)$ we iterate again the previous expressions combined with Lemma \ref{BarridosNoAfectan} to conclude that the monodromy maps $b_3$ as follows (where we use again that $W_K$ maps to itself for the same reason stated in Remark \ref{HolonomiesTangleSmall}): 

\begin{align*}
b_3 \mapsto & W_{K}+(W_{K}+b_1b_2W_{K})(1+b_2b_3)W_{K}=\\ & W_K+W_K^2+W_K b_2b_3W_K+b_1b_2W_K^2+b_1b_2W_Kb_2b_3W_K.
\end{align*}

Observe that $\tau_{b_3}\left(\mu(b_3)-b_3\right)=2\ell(W_K)^2+1$ is not an even number so Lemma \ref{PropiedadesIdealDiferencial} implies that $[\mu(b_3)-b_3]\in H_0(K_{p,q}\#K_{3,2})$ is not zero. This concludes the argument for $j=3$ and completes the proof.
\end{proof}

In the next Section we  provide infinitely many examples for which Theorem \ref{thm:KalmanSumNonTrivialA}\label{familia1} applies.

\section{Infinite families of loops with non-trivial LCH Monodromy}\label{SecExamples}

In this Section we construct infinitely many new examples of loops of Legendrians with non-trivial monodromy invariant. More specifically, we  provide three essentially different infinite families of loops, each constructed by using the parametric connected-sum operation.
Making use of the machinery developed throughout the article, we will conclude that all the examples in each of the families has non-trivial LCH monodromy invariant.

\subsection{First family of examples via parametric connected-sums.}

Consider a connected sum of $m$ Legendrian positive torus knots of even $\partial$-class $\tilde{K}:=K_1\#\cdots K_k$ (recall that $\tilde{K}$ is then of even $\partial$-class as well by Corollary \ref{ConnectedSumsN2}). %Just to keep a particular example in mind, one could think, for instance, of a connected sum of $m$ positive torus knots of the form $\tilde{K}:=K_{n_1, 2}\#\cdots\# K_{n_k,2}$ with $n_i\not\equiv 2 \mod 3$.

%Take the connected sum of $m$ torus knots, $\tilde{K}=K_{p_1, q_1}\#\cdots\# K_{p_m,q_m}$, so that $K=\tilde{K}\# K_{3,2}$ is of even $\partial$-class (see Figure \ref{concatenationLoops}). Corollary \ref{ConnectedSumsN2} implies that there are infinitely many examples of such $\tilde{K}$. 

For each $i\in\{1,\ldots,m\}$, choose a loop $\mathcal{K}_i(t,\theta)$ based at $K_{i}$. For instance $\mathcal{K}_i$ can be taken to be any (positive or negative) power of \kalmans loop, the constant loop or a Lagrangian rotation (recall Figure \ref{Fig:LoopBoxRotation}). Then, it follows by Theorem \ref{thm:KalmanSumNonTrivialA} that any loop in the following family does not have trivial monodromy invariant:

\begin{equation}\label{FamilyOfLoops}
\bigg\lbrace{\mathcal{K}_1^{\ell_1}\#\cdots\#\mathcal{K}_{k}^{\ell_k}\#\LL_{3,2}^j(t,\theta):\quad l_i\in\Z, \quad  j\in\{1,2,3\}\bigg\rbrace}.
\end{equation}

This geometrically corresponds to the loop defined iteratively on each tangle (see Figure \ref{concatenationLoops}). The loop $\mathcal{K}_1^{\ell_1}$ is performed first on the first component playing the role of the elephant while the rest of the tangles play the role of the fly. Then, that component gets shrunk and the next component gets inflated so that it plays now the elephant role. This process is iteratively carried out until the last component, corresponding to the trefoil, which performs the loop $ \LL_{3,2}^j(t,\theta)$.

\begin{figure}[h]
	\centering
	\includegraphics[width=0.4\textwidth]{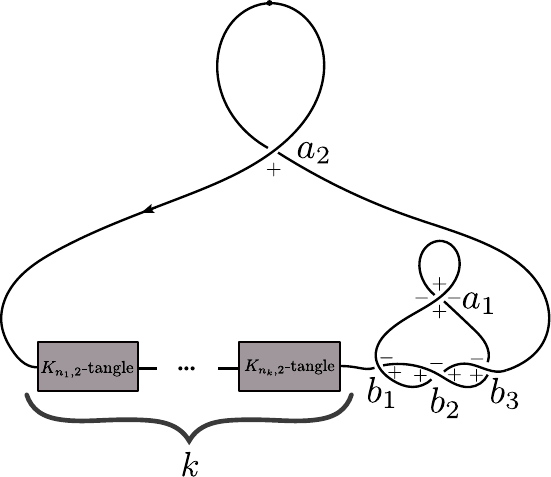}
	\caption{The Legendrian $K=\tilde{K}\#K_{3,2}$ with the labelling of some crossings.}\label{concatenationLoops} 
\end{figure}

%Note that by varying $m$, $\ell_i$ (for each $i=1,\cdots, m$) and the chosen loops $\mathcal{K}_i(t,\theta)$ we get infinitely many combinations of different loops of Legendrians. All of them have non-trivial monodromy invariant, as we showed.

\subsection{Second family of examples via parametric connected-sums.}\label{familia2}

We will now describe an essentially different infinite family of examples. The substantial difference is that we will consider interactions between the different $K_{i}$-tangles given by the Legendrian ``pulling one knot through another'' loops (we refer to the work of R. Budney \cite[Sec. 4]{Budney2}). These loops can be geometrically described as follows. 

Consider an arbitrary connected sum of $m$ knots of even $\partial$-class, $\tilde{K}:=K_1\#\cdots\# K_m$. Once again, note that $K:=\tilde{K}\# K_{3,2}$ is of even $\partial$-class by Proposition \ref{ConnectedSumsEvenDeltaClass}.

Consider two adjacent tangles $K_i$ and $K_{i+1}$ in $\tilde{K}$. We can shrink $K_i$ arbitrarily and pass it all along the component of $K_{i+1}$ until they exchange positions. After that, we can shrink $K_{i+1}$, enlarge $K_i$ and repeat the process with their roles interchanged. See Figure \ref{PullThrough} for an example. It follows from how these loops are defined that they are loops of Legendrian embeddings in standard position.

\begin{figure}[h]
	\centering
	\includegraphics[width=0.8\textwidth]{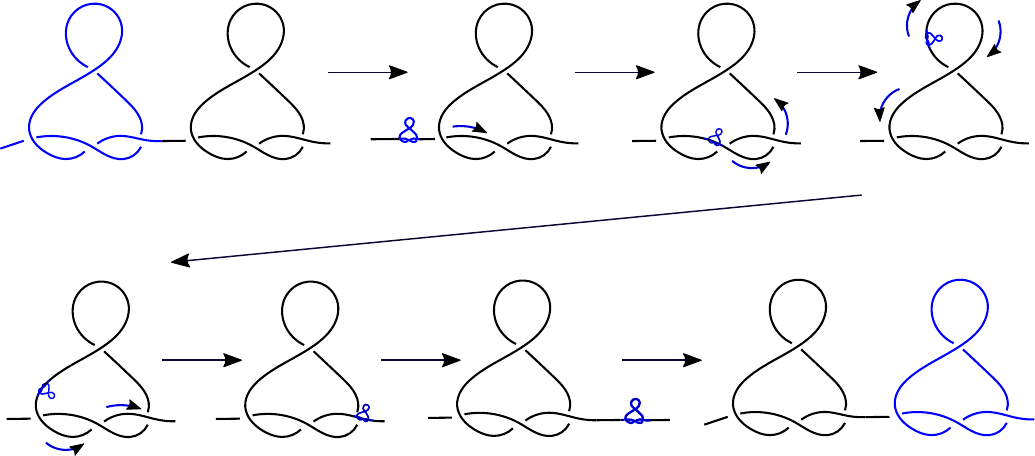}
	\caption{First half of the ``Legendrian pulling one knot through another'' loop for two Legendrian trefoils. In the second half, the roles are exchanged and the black trefoil shrinks and passes through the blue trefoil. }\label{PullThrough}
\end{figure}

Consider arbitrary compositions of all the possible combinations (by taking into considerations all possible adjacent tangles $K_i$ and $K_{i+1}$) of these ``Legendrian pulling one knot through another'' loops. Each set of combinations gives raise to a loop $\mathcal{K}(t,\theta)$  which we can connect-sum to any of the loops $\LL_{3,2}^j$, $j=1,2,3$. All loops $\mathcal{K}\#\LL_{3,2}^j(t,\theta)$ defined in this fashion yield loops with non-trivial monodromy invariant by Theorem \ref{thm:KalmanSumNonTrivialA}.

\begin{remark}\label{NotInvolveTrefoil}
In order to apply Theorem \ref{thm:KalmanSumNonTrivialA}, the ``pulling one knot through another'' loops cannot involve the last $K_{3,2}$-tangle (i.e. the right-most one in the diagram). Indeed, note that pulling some $K_{i}$-tangle through the last trefoil could potentially produce non-trivial holonomies affecting the $0$-degree crossings of the $K_{3,2}$-tangle.
\end{remark}

\begin{remark}
    Note that if both tangles are equal (as in Figure \ref{PullThrough}), the first part of the loop already defines an actual loop of Legendrian knots by itself.
\end{remark}

This family of examples is, in fact, a particular instance of a more general pattern. We can produce even more involved examples by allowing more general interactions between the tangles as follows. 

\textbf{Loops coming from transpositions in $S_m$}. Consider the symmetric group $S_m$ corresponding to permutations of $m$ elements. Identify a particular (non-unique) isotopy that exchanges tangle-$i$ and tangle-$i+1$ with the transposition $(i\quad i+1)\in S_m$. Note that there are a priori more than one possible such isotopies, e.g. the one where tangle-$i$ shrinks and passes through tangle-$i+1$ and the one where the roles are reversed. Then any product of such transpositions in $S_m$ which equals the identity element $Id\in S_m$ (i.e. forms a cycle) gives raise, via such identification, to a loop $\mathcal{K}(t,\theta)$ of Legendrians based at $\tilde{K}$. Again, the monodromy invariant of all loops $\mathcal{K}\#\LL^j_{3,2}(t,\theta)$ described in this fashion will be non trivial for $j=1,2,3$ by Theorem \ref{thm:KalmanSumNonTrivialA}.

\subsection{Third family of examples via parametric connected-sums.}\label{ThirdFamily}

Finally, we can consider a family of loops that combines both parametric connected sums of individual loops in an analogous manner to the first family, together with loops coming from transpositions by the identification above. Precisely, 
consider an arbitrary connected sum of $m$ knots of even $\partial$-class, $\tilde{K}:=K_1\#\cdots\# K_m$. One more time, we have that the knot $K:=\tilde{K}\# K_{3,2}$ is of even $\partial$-class by Proposition \ref{ConnectedSumsEvenDeltaClass}.

For each $i\in\{1,\ldots,m\}$, choose a loop $\mathcal{K}_i(t,\theta)$ based at $K_i$. For example $\mathcal{K}_i$ can be taken to be the constant loop, a Lagrangian rotation (recall Figure \ref{Fig:LoopBoxRotation}) or any (positive or negative) power of \kalmans loop if $K_i$ is a positive torus knot, although any preferred loops could  be considered. Then, we can define a whole family of loops as in Expression (\ref{FamilyOfLoops}), but this time allowing interactions between the tangles (not involving the last trefoil; recall Remark \ref{NotInvolveTrefoil}) by considering loops coming from transpositions as in the previous family of examples. 

These mixed combinations can be performed in any preferred order, giving rise to new infinite families of loops. Again, by Theorem \ref{thm:KalmanSumNonTrivialA}, all these loops have non-trivial monodromy invariant. See Figure \ref{loopsMix} for the depiction of a particular example (among the infinitely many described possibilities) where three positive right-handed trefoils are considered. Nonetheless, see the following Remark regarding this particular depiction.

\begin{remark}
    As pointed out in Remark \ref{RemarkSP}, each time that a loop of embeddings in standard position is considered in the construction, we can represent it as taking place simultaneuosly to the last loop involving the rightmost trefoil. For instance, the loop in Figure \ref{ExampleLoopIntro} involves a ``pulling one knot through another'' loop and a ``Kálmán's loop'' that are taking place simultaneuosly (recall Remark \ref{ExampleLoopIntro}). Recall that such a loop is isotopic to the one where they take place in order one at a time (Remark \ref{RemarkSP}) and so they have the same non-trivial LHC monodromy (Thm. \ref{MonodromyInvariant}). \end{remark}

\end{document}